\documentclass[12pt]{article}
\usepackage{latexsym,amsmath,amsthm,amssymb,amscd,amsfonts}
\usepackage[usenames]{color}

\newtheorem{lemma}{Lemma}[section]
\newtheorem{proposition}[lemma]{Proposition}
\newtheorem{remark}[lemma]{Remark}
\newtheorem{theorem}[lemma]{Theorem}
\newtheorem{definition}[lemma]{Definition}

\newtheorem{corollary}[lemma]{Corollary}

\newtheorem*{remark*}{Remark}
\newtheorem*{theorem*}{Theorem}

\def\cH{{\cal H}}
\def\H{\text{H}}
\def\R{{\mathbb R}}

\def\eps{\varepsilon}

\def\Ham{\text{Ham}}

\def\supp{\text{supp}}

\def\cD{\mathcal{D}}
\def\od{\overline{d}}

\def\cG{\mathcal{G}}
\def\cA{\mathcal{A}}
\def\cJ{\mathcal{J}}

\def\diam{\text{diam}}

\def\OHam{\overline{\text{Ham}}}
\def\Homeo{\text{Homeo}}
\def\Hameo{\text{Hameo}}
\def\FHomeo{\text{FHomeo}}
\def\osc{\text{osc}}

\newcommand{\floor}[1]{\left\lfloor #1 \right\rfloor}
\newcommand{\ceil}[1]{\left\lceil #1 \right\rceil}

\newcommand{\Id}{{{\mathchoice {\rm 1\mskip-4mu l} {\rm 1\mskip-4mu l}
      {\rm 1\mskip-4.5mu l} {\rm 1\mskip-5mu l}}}}

\makeatletter \@addtoreset {equation}{section}

\renewcommand\theequation
  {\ifnum \c@subsection>\z@ \arabic{section}.\arabic{subsection}.\arabic{equation}
  \else \arabic{section}.\arabic{equation} \fi}
\makeatother

\begin{document}

\title{On two remarkable groups of area-preserving homeomorphisms}

\author{Lev Buhovsky}%$^{1}$}

%\footnotetext[1]{The author also uses the spelling ``Buhovski"
%for his family name.}

\date{\today}
\maketitle

\begin{abstract}
We prove that on a symplectic sphere, the group of Hamiltonian homeomorphisms in the sense of Oh and M\"uller is a proper normal subgroup of the group of finite energy Hamiltonian homeomorphisms. Moreover we detect infinite-dimensional flats inside the quotient of these groups endowed with the natural Hofer pseudo-metric. 
\end{abstract}

%\noindent

\section{Introduction}

Let $ (M,\omega) $ be a closed and connected symplectic surface. Consider the groups $ \Ham(M,\omega) \subset  \Homeo(M) $ of Hamiltonian diffeomorphisms of $ (M,\omega) $, and of all homeomorphisms of $ M $, respectively. The $ C^0 $ closure of $ \Ham(M,\omega) $ inside $ \Homeo(M) $ is denoted by $ \OHam(M,\omega) $, and it consists exactly of those orientation and area preserving homeomorphisms of $ (M,\omega) $ that are homotopic to the the identity and which belong to the kernel of the mass-flow homomorphism \cite{F}. In the genus $ 0 $ case, i.e. when $ M = S^2 $, the condition of vanishing of the mass-flow is redundant, and so $ \OHam(S^2) $ consists of all orientation and area preserving homeomorphisms. The group $ \OHam(M,\omega) $ can also be defined for a general $ M $, under additional compact support requirements. 

The question of Fathi \cite{F} asks whether $ \OHam(M,\omega) $ is a simple group, and it has been one of the major inspirations for the development of $ C^0 $ symplectic geometry. The Fathi question served as an important motivation for the influential work \cite{OM} of Oh and M\"uller which in particular introduced the notion of a continuous Hamiltonian flow on symplectic manifolds. In the $ 2 $-dimensional case, the group $ \Hameo(M,\omega) $ consisting of time-$1$ maps of these flows was conjectured \cite{OM} to be an example of a proper normal subgroup of $ \OHam(M,\omega) $. 

Recently, a number of breakthrough works addressed the Fathi question, where it was first solved in the case of a two-disc \cite{CHS-1}, then in the two-sphere case \cite{CHS-2,PS}, and finally for general surfaces of finite type and finite area \cite{CHMSS}. Moreover, the works \cite{CHS-2,PS} have largely contributed to Hofer geometry, in particular solving the Polterovich-Kapovich question. Powerful tools coming from Floer homology, embedded contact homology and periodic homology theories, were central in making that progress possible. The aim of the present article is to give an additional insight on that picture. Our approach is based on application of novel Floer-theoretic invariants from \cite{PS} (see also \cite{MS}), in combination with a soft approach relying on an idea due to Sikorav \cite{Si}.

%A recent resolution of the Fathi conjecture \cite{F} from topological surface dynamics and of the Kapovich-Polterovich question {\lev{CITE}} from Hofer's geometry, as well as an important progress in Lagrangian packing {\lev{CITE}} has been a major breakthrough. The powerful tools coming from embedded contact homology and periodic homology theories {\lev{CITE}}, as well as novel fascinating Floer-theoretic tools that were recently introduced {\lev{CITE}}, were central in making that progress possible. 
%
%This question was addressed in a number of recent works \lev{CITE} where the question was partially solved 
%fully resolved in \lev{CITE-CHSMS}, where the ideas were inspired by a number of previous 
%
%These results are motivated by the recent major progress in topological surface dynamics and Hofer's geometry: the resolution of the Fathi and the Polterovich-Kapovich questions. That progress became possible largely due to appearance of new families of symplectic invariants serving as powerful tools in addressing these and other important questions. Our approach is based on application of such invariants, in combination with a soft approach based on a refinement of Sikorav's trick/lemma.
%
%
%
%Let us briefly recall the statement of the Fathi conjecture and its history. Let $ (M,\omega) $ be a compact and connected symplectic surface, possibly with boundary. 

Before discussing our results, we recall some of the relevant definitions. We refer the reader to Section \ref{section:Notation} for other definitions and notation that we use. 
Let $ (M,\omega) $ be a symplectic manifold. We denote by $ \Homeo(M) $ the group of all compactly supported homeomorphisms of $ M $. 
The $ C^0 $ convergence of a sequence $ \phi_k \in \Homeo(M) $ to some $ \phi \in \Homeo(M) $ always assumes that the supports of all $ \phi_k $ lie in some compact subset of $ M $, and is denoted by $ \phi = (C^0) \lim_{k \rightarrow \infty} \phi_k $. For a continuous compactly supported function $ H : M \rightarrow \R $, its $ L^\infty $ norm is denoted by $ \| H \| $. For a continuous compactly supported function $ H : M \times [0,1] \rightarrow \R $, its $ L^{(1,\infty)} $ norm is given by $$ \| H \|_{(1,\infty)} = \int_0^1 \| H_t \| \, dt ,$$
where $ H_t (\cdot) = H(\cdot,t) $. All Hamiltonian functions are assumed to be compactly supported. 

On a closed symplectic manifold $ (M,\omega) $, for any $ \phi \in \Ham(M,\omega) $, its Hofer norm is given by
$$ \| \phi \|_\H = \inf \| H \|_{(1,\infty)} ,$$
where $ H $ is a smooth and normalized Hamiltonian function such that $ \phi = \phi_H^1 $.
We denote by $ d_\H $ the Hofer distance on $ \Ham(M,\omega) $, that is, $ d_\H(\phi,\psi) = \| \phi^{-1} \psi \|_\H $.

Now we recall the definitions of the groups $ \OHam(M,\omega) $ (Hamiltonian homeomorphisms), $ \Hameo(M,\omega) $ (strong Hamiltonian homeomorphisms \cite{OM}) and $ \FHomeo(M,\omega) $ (finite energy Hamiltonian homeomorphisms \cite{CHS-1,CHS-2,CHMSS}).  

\begin{definition}
A homeomorphism $ \phi \in \Homeo(M) $ is a Hamiltonian homeomorphism if it is a $ C^0 $ limit of a sequence of Hamiltonian diffeomorphisms of $ M $. The set of all Hamiltonian homeomorphisms of $ M $ is denoted by $ \OHam(M,\omega) $.
\end{definition}

\begin{definition}
A continuous path $ (\phi^t)_{t \in [0,1]} $ of homeomorphisms of $ M $ is a {\em continuous Hamiltonian flow}, if there exists a continuous compactly supported function $ H : M \times [0,1] \rightarrow \R $, and a sequence $ H_k : M \times [0,1] \rightarrow \R $ of smooth Hamiltonian functions, such that:
\begin{itemize}
 \item The supports of $ H_k $ all belong to some compact subset of $ M $.
 \item $ \lim_{k \rightarrow \infty} \| H_k - H \|_{(1,\infty)} = 0 $.
 \item We have the $ C^0 $ convergence $ \phi^t = (C^0) \lim_{k \rightarrow \infty} \phi_{H_k}^t $, uniform in $ t \in [0,1] $.
\end{itemize}
In that case the time-$1$ map $ \phi^1 $ is called a strong Hamiltonian homeomorphism of $ M $, and the set of all such homeomorphisms is denoted by\footnote{In fact, the original definition 
of $ \Hameo(M,\omega) $ was through the notion of a {\em topological Hamiltonian path} \cite{OM}. However, as it was shown in \cite{Mu}, the two definitions are equivalent.}
 $ \Hameo(M,\omega) $.
\end{definition}

\begin{definition}
An element $ \phi \in \OHam(M,\omega) $ is a {\em finite energy homeomorphism} if there exists a sequence of smooth Hamiltonians $ H_k : M \times [0,1] \rightarrow \R $ such that 
$$ \phi = (C^0) \lim_{k \rightarrow \infty} \phi_{H_k}^1 , \,\,\,\, \text{and} \,\,\,\,\,\,  \| H_k \|_{(1,\infty)} \leqslant C ,$$ 
for some constant $ C > 0 $. The set of all finite energy Hamiltonian homeomorphisms of $ M $ is denoted by $ \FHomeo(M,\omega) $.
\end{definition}

It is well known \cite{OM,CHMSS} that $ \Hameo(M,\omega) \subset \FHomeo(M,\omega) $ are normal subgroups of $ \OHam(M,\omega) $. For symplectic surfaces $ (M,\omega) $ of finite type and area it was shown \cite{CHS-1,CHS-2,PS,CHMSS} that $ \FHomeo(M,\omega) $ is in fact a {\em proper} normal subgroup of $ \OHam(M,\omega) $. This answers the 
Fathi question to the negative. This of course also means that $ \Hameo(M,\omega) $ is a proper normal subgroup of $ \OHam(M,\omega) $, confirming the prediction of Oh and M\"uller.
However, a natural question remained whether $ \Hameo(M,\omega) $ and $ \FHomeo(M,\omega) $ are distinct groups. This question was communicated to us by V. Humili\`ere and S. Seyfaddini. Below we show a positive answer to the question in the case of $ S^2 $ (Corollary \ref{cor:main}).

Let $ (M,\omega) $ be a closed symplectic manifold. The group $ \FHomeo(M,\omega) $ naturally carries a Hofer like norm and the associated metric which we denote here by $ \| \cdot \|_\cH $ and $ d_\cH $ respectively \cite{CHS-2}. For any $ \phi \in \FHomeo(M,\omega) $, its norm $ \| \phi \|_\cH $ is defined as the minimal possible $ \liminf_{k \rightarrow \infty} \| \phi_k \|_\H $, where $ (\phi_k) $ is a sequence in $ \Ham(M,\omega) $ that $ C^0 $ converges to $ \phi $. It can be easily verified that the norm $ \| \cdot \|_\cH $ is invariant under conjugation by elements of $ \OHam(M,\omega) $. It is currently completely unknown whether the norm $ \| \cdot \|_\cH $ coincides with $ \| \cdot \|_\H $ on $ \Ham(M,\omega) $ (a question of Le Roux \cite{L}). 

Consider the symplectic two-sphere $ (S^2,\omega) $ of area $ 1 $. Our first result is:
\begin{theorem} \label{thm:main}
For every $ E > 0 $, there exists a continuous path $ (\varphi^t)_{t\in[0,1]} $ of homeomorphisms of $ S^2 $, such that:
\renewcommand\labelenumi{(\theenumi)}
\begin{enumerate}
 \item The flow $ (\varphi^t) $ is the uniform limit of a sequence of smooth Hamiltonian flows $ (\phi_{H_k}^t) $, where $ H_k \in C^\infty(S^2 \times [0,1]) $ is normalized and $ \| H_k \|_{(1,\infty)} \leqslant E $ for every $ k $. 
 \item For every sequence $ \psi_k \in \Ham(S^2) $ satisfying $ \varphi^1 = (C^0) \lim_{k\rightarrow\infty} \psi_k $, and for every $ \psi \in \Ham(S^2) $ we have 
 $ \liminf_{k \rightarrow \infty}  d_\H(\psi_k,\psi) \geqslant E $.
\end{enumerate}
Moreover, the flow $ (\varphi^t) $ can be chosen to be arbitrarily $ C^0 $-close to any given smooth Hamiltonian flow in $ \Ham(S^2) $ of Hofer's length $ \leqslant E $. 
\end{theorem}

If $ (\varphi^t) $ is a path of homeomorphisms of $ S^2 $ given by Theorem \ref{thm:main}, then by the item $(1)$ from the theorem, the time-$1$ map $ \varphi = \varphi^1 $ is an element of $ \FHomeo(S^2) $. Both $(1)$ and $(2)$ imply that on the one hand we have $ \| \varphi \|_{\cH} = E $, but on the other hand the $ d_\cH $-distance of $ \varphi $ to any element of $ \Ham(S^2) $ is greater than or equal to $ E $. The latter property yields $ \varphi \notin \Hameo(S^2) $. Indeed, assuming the contrary, that is $ \varphi \in \Hameo(S^2) $, we obtain a sequence $ H_k' : S^2 \times [0,1] \rightarrow \R $ of smooth Hamiltonian functions converging in the $ L^{(1,\infty)} $ norm to a continuous function $ H' : S^2 \times [0,1] \rightarrow \R $, such that in particular we have $ \varphi = (C^0) \lim_{k\rightarrow\infty} \phi_{H_k'}^1 $. But then we come to a contradiction with the item (2) of Theorem \ref{thm:main}, since we can consider the sequence $ \psi_k := \phi_{H_k'}^1 $ and then take $ \psi := \phi_{H_l'}^1 $ for $ l $ large enough. Thus we obtain:

\begin{corollary} \label{cor:main}
$ \Hameo(S^2) \neq \FHomeo(S^2) $.
\end{corollary}

This answers the mentioned above question. $ \Hameo(S^2) $ is a normal subgroup of $ \FHomeo(S^2) $, and a next possible task is to check how large the quotient group $ \FHomeo(S^2) / \Hameo(S^2) $ is, and to understand its algebraic structure. Moreover, the norm $ \| \cdot \|_\cH $ naturally descends from $ \FHomeo(S^2) $ to the pseudo-norm $ \| \cdot \|_\cH $ on the quotient $ \cG = \FHomeo(S^2) / \Hameo(S^2) $. It might be interesting to understand that picture better. These questions will be discussed in Section \ref{section:The-quotient} below.

In our proof of Theorem \ref{thm:main} we use as a tool the novel powerful versions of Lagrangian spectral estimators developed in \cite{PS}, which are certain functionals defined on the space of time dependent Hamiltonians (spectral estimators), and on the group of Hamiltonian diffeomorphisms (group estimators). These functionals and their relatives were used in \cite{PS} to show several remarkable applications - in Hofer's geometry, Lagrangian packing, and $ C^0 $ symplectic geometry. Let us briefly describe them and some of their basic properties. 

As before, consider the symplectic sphere $ (S^2,\omega) $, where $ \omega $ is normalized by $ \omega(S^2) = 1 $, and think of it as sitting inside $ \R^3 $ as the sphere of radius $ \frac{1}{2} $ centered at the origin, equipped with the standard area form divided by $ \pi $. Denote by $ x_3 : S^2 \rightarrow \mathbb{R} $ the $x_3$-coordinate function, let $ k \geqslant 1 $ be an integer, and pick a pair of positive real numbers $ 0 < C < B $ satisfying $ 2B + (k-1)C = 1 $. For $ 0 \leqslant j < k $ denote $ L_{k,B}^{0,j} = (x_3)^{-1}(-1/2 + B + jC) $. This gives us a finite collection of ``horizontal'' circles on $ S^2 $, which are of course Lagrangian submanifolds. Then there exists a map $ c_{k,B}^0 : C^\infty(S^2 \times [0,1]) \rightarrow \R $ with the following properties \cite{PS} (some properties described in $ \cite{PS} $ are omitted since we will not use them in the sequel):

\begin{enumerate}
 \item {\em (Hofer-Lipschitz)} For each $ G,H \in C^\infty(S^2 \times [0,1])$,
          $$  |c_{k,B}^0(G) - c_{k,B}^0(H)| \leqslant \int_0^1 \max |G_t-H_t| \, dt   .$$
          %where $ G_t(x) = G(x,t) $ and $ H_t(x) = H(x,t) $.
 \item {\em (Monotonicity)} If $ G,H \in C^\infty(S^2 \times [0,1]) $ satisfy $ G \leqslant H $ as functions, then 
          $$ c_{k,B}^0(G) \leqslant c_{k,B}^0(H) .$$
 \item {\em (Normalization)} For each $ H \in C^\infty(S^2 \times [0,1]) $ and $ b \in C^\infty([0,1]) $,
          $$ c_{k,B}^0(H+b) = c_{k,B}^0(H) + \int_0^1 b(t) \, dt .$$ 
 \item {\em (Lagrangian control)} For any $ H \in C^\infty(S^2 \times [0,1]) $ 
          such that $ (H_t)|_{L_{k,B}^{0,j}} \equiv c_j(t) $ for all $ 0 \leqslant j < k $, we have
          $$ c_{k,B}^0(H) = \frac{1}{k} \sum_{0 \leqslant j < k} \int_0^1 c_j(t) \, dt . $$ 
 \item {\em (Independence of Hamiltonian)} For a normalized Hamiltonian $ H \in C^\infty(S^2 \times [0,1]) $, the value 
         $$ c_{k,B}^0(H) = c_{k,B}^0(\phi_H^1) $$
         depends only on the time-$1$ map $ \phi_H^1 \in \Ham(S^2) .$
 \item {\em (Subadditivity)} For all $ \phi,\psi \in \Ham(S^2) $, 
         $$ c_{k,B}^0(\phi \psi) \leqslant c_{k,B}^0 (\phi) + c_{k,B}^0 (\psi) .$$
 \item {\em ($C^0$-continuity)} The map 
                                                  $$ \tau_{k,k',B,B'} : \Ham(S^2) \rightarrow \R, $$
                                                  $$ \tau_{k,k',B,B'} = c_{k,B}^0 - c_{k',B'}^0 $$ 
                                                  is $2$-Lipschitz in Hofer's metric, is $ C^0 $-continuous, and extends to $ \OHam(S^2) $ by continuity. 
\end{enumerate}
%Another ingredient in our approach is Lemma \ref{lemma:Sikorav-style} from Section \ref{section:Proofs} below, which looks similar to the so-called Sikorav's lemma. However, to the best of our knowledge, the particular statement of Lemma \ref{lemma:Sikorav-style} did not appear in the literature before. 

Recall that the above Lagrangian control property of the functionals $ c_{k,B}^0 $ is associated with the collection $ \{ L_{k,B}^{0,j} \}_{0 \leqslant j < k} $ of horizontal circles on $ S^2 $. The work \cite{PS} by Polterovich and Shelukhin has inspired a later work \cite{CHMSS} by Cristofaro-Gardiner, Humili\`ere, Mak, Seyfaddini and Smith, which introduced new invariants sharing properties similar to these of $ c_{k,B}^0 $. These invariants from \cite{CHMSS} can be defined on general closed symplectic surfaces $ (M,\omega) $, and moreover the corresponding Lagrangian control property for them holds for quite general collections of circles on $ M $. The invariants are denoted by $ c_{\underline{L}} $, where $ \underline{L} $ is a suitable given collection of circles on $ M $. In particular, one may choose a collection $ \underline{L} $ so that most of the circles in it (all except for a restricted number of them) are small circles bounding discs of the same area which are ``spread uniformly'' over the surface and capture almost all of its area. That particular property of $ \underline{L} $ was crucial for deriving the so-called ``Calabi property'' for the $ c_{\underline{L}} $'s \cite[page 3, Theorem 1.1]{CHMSS}, which in turn has remarkable applications such as the Fathi conjecture \cite{CHMSS}. We wish to remark that by using the invariants $ c_{\underline{L}} $ from \cite{CHMSS} instead of $ c_{k,B}^0 $ from \cite{PS}, one can extend Theorem \ref{thm:main} to all closed symplectic surfaces.

%, and moreover simplify its proof since then Lemma \ref{lemma:Sikorav-style} is not needed anymore. 
%However, we still preferred to keep the statement and proof of Theorem \ref{thm:main} as presented. The proof does not require the full Lagrangian control property but only its ``small part'', as the usage of Lemma \ref{lemma:Sikorav-style} demonstrates. Although there are notable differences in the approaches of \cite{PS} and \cite{CHMSS}, and in particular in the constructions of the invariants $ c_{k,B}^0 $ and $ c_{\underline{L}} $ respectively, Lemma \ref{lemma:Sikorav-style} and the way it is used in the proof of Theorem \ref{thm:main} still enable to imagine a possible connection between them, at least on a philosophical level. But more concretely, Lemma \ref{lemma:Sikorav-style} can be applied to deduce the Calabi property under relatively weak assumptions. Moreover, Lemma \ref{lemma:Sikorav-style} combined with the Calabi property for \lev{[ADD REFERENCE TO POLTEROVICH-SHELUKHIN (UNPUBLISHED)]} the functionals $ c_{k,B}^0 $ brings some insight on the geometric group structure of the quotient $ \FHomeo(S^2) / \Hameo(S^2) $ (see Section \ref{section:The-quotient}). 

%\subsection{Further discussion} \label{section:Further}

\subsection{The Calabi property} \label{subsection:Calabi}

The mentioned above Calabi property for the invariants $ c_{\underline{L}} $ states the following \cite{CHMSS}:
\begin{theorem*}
Let $ \underline{L}^m $ be a sequence of equidistributed Lagrangian links in a closed symplectic surface $ (M,\omega) $. Then, for any $ H \in C^\infty(M \times [0,1]) $ we have
$$ \lim_{m \rightarrow \infty} c_{\underline{L}^m} (H) = (\omega(M))^{-1} \int_0^1 \int_M H_t\omega \, dt .$$
\end{theorem*}
Here a Lagrangian link means a collection of pairwise disjoint smoothly embedded circles, and for the precise definition of a sequence of equidistributed Lagrangian links we refer the reader to \cite{CHMSS}.
We claim, however, that in principle one can relax the assumptions and still have the Calabi property, so that in particular, relying on the ``equidistribution property'' in unnecessary. In the next proposition, $ \widetilde{\Ham}(M,\omega) $ stands for the universal cover of the Hamiltonian group $ \Ham(M,\omega) $.

\begin{proposition} \label{prop:Calabi-general}
Let $ (M,\omega) $ be a closed and connected symplectic surface and let $ c_m : C^\infty(M \times [0,1]) \rightarrow \R $ be a sequence of functionals, satisfying:
\begin{enumerate}
 \item {\em (Hofer-Lipschitz)} For each $ G,H \in C^\infty(M \times [0,1]) $,
          $$  |c_{m}(G) - c_{m}(H)| \leqslant \int_0^1 \max |G_t-H_t| \, dt   .$$
 \item {\em (Monotonicity)} If $ G,H \in C^\infty(M \times [0,1]) $ satisfy $ G \leqslant H $ as functions, then 
          $$ c_{m}(G) \leqslant c_{m}(H) .$$
 \item {\em (Normalization)} For each $ H \in C^\infty(M \times [0,1]) $ and $ b \in C^\infty([0,1]) $,
          $$ c_{m}(H+b) = c_{m}(H) + \int_0^1 b(t) \, dt .$$ Moreover, $ c_m({\bf 0}) = 0 $ where $ {\bf 0} : M \times [0,1] \rightarrow \R $ is the zero function.
 \item {\em (Independence of Hamiltonian)} For a normalized Hamiltonian $ H \in C^\infty(S^2 \times [0,1]) $, the value 
         $$ c_m(H) = c_m(\phi_H) $$
         depends only on the class $ \phi_H = [(\phi_H^t)] \in \widetilde{\Ham}(M,\omega) .$
 \item {\em (Subadditivity)} For any $ \phi, \psi \in \widetilde{Ham}(M,\omega) $,
         $$ c_{m}(\phi \psi) \leqslant c_{m}(\phi) + c_{m}(\psi) .$$
 \item {\em (Locality)} For each $ m $ there exists a smooth function $ h_m : M \rightarrow [0,1] $ and an open topological disc $ \cD_m \subset M $, such that:
  \begin{itemize}
   \item $ \supp (h_m) \subset \cD_m $, 
   \item The diameter of $ \cD_m $ converges to $ 0 $ when $ m \rightarrow \infty $.
   \item $ \lim_{m \rightarrow \infty} \left( \omega(M) / \omega(\cD_m) \right) c_m(-h_m) = - 1 $.
  \end{itemize}
\end{enumerate}
Then for every $ H \in C^\infty(M \times [0,1]) $ we have
\begin{equation} \label{eq:Calabi-property}
 \lim_{m \rightarrow \infty} c_{m} (H) = (\omega(M))^{-1} \int_0^1 \int_M H_t\omega \, dt .
\end{equation}
\end{proposition}

%The proof of Theorem \ref{prop:Calabi-general} reminds our proof of Lemma \ref{lemma:main}. 

%\begin{remark} \label{remark:Calabi}
\subsubsection{Some remarks} \label{remark:Calabi}
%\begin{enumerate}
          Note that the Normalization, Independence of Hamiltonian and Subadditivity properties for $ c_m $ in Proposition \ref{prop:Calabi-general} are stated differently from the ones which were 
          previously stated for the invariants $ c_{k,B}^0 $. Also, the Locality property combined with the normalization condition $ c_m({\bf 0}) = 0 $ in some sense replace the previously stated 
          Lagrangian control property for $ c_{k,B}^0 $. The main reason to relax the Independence of Hamiltonian is that other than $ c_{k,B}^0 $ relevant functionals introduced in \cite{PS} and 
          \cite{CHMSS} do not necessarily satisfy the original version of the Independence of Hamiltonian property.  
 
          From the Monotonicity property for $ c_m $ it follows that we can extend the functional to all bounded functions $ H $. Concentrating on autonomous functions, for every
          bounded function $ H: M \rightarrow \R $ we may define $ c^+_m(H) := \inf c_m(G) $, where $ G \in C^\infty(M) $ satisfies $ G \geqslant H $ on $ M $. Then in terms of $ c^+_m $, 
          the Locality property can be reformulated by saying that for each $ m $ there exists an open topological disc $ \cD_m \subset M $, such that $ \diam(\cD_m) \rightarrow 0 $ when 
          $ m \rightarrow \infty $, and such that 
          \begin{equation} \label{eq:equiv-formulation-locality}
           \lim_{m \rightarrow \infty} \left( \omega(M) / \omega(\cD_m) \right) c_m^+(-\Id_{\cD_m}) = - 1 , 
          \end{equation} 
          where $ \Id_{\cD_m} : S^2 \rightarrow \R $ 
          is the characteristic function of the disc $ \cD_m $. Indeed, if we assume $(\ref{eq:equiv-formulation-locality})$ then the Locality property clearly follows by the definition 
          of $ c^+_m $ and by the Monotonicity and Normalization properties. For the other direction, assuming the Locality property, we immediately get that the left hand 
          side of $(\ref{eq:equiv-formulation-locality})$ is not greater than $ -1 $. However, by applying Lemma \ref{lemma:Sikorav-style} similarly as in the proof of 
          Proposition \ref{prop:Calabi-general} (or more precisely, in the proof of Lemma \ref{lemma:main} on which Proposition \ref{prop:Calabi-general} relies), one can deduce 
          that the opposite inequality always holds. 

          As a corollary of Proposition \ref{prop:Calabi-general} we conclude that the functionals $ c_{k,B}^0 $ satisfy the Calabi property: 
          
          \begin{theorem} \label{thm:Calabi-for-c0kB}
           For a sequence $ B_k \in (\frac{1}{k+1}, \frac{1}{2}) $ that converges to $ 0 $, for any 
           $ H \in C^\infty(S^2 \times [0,1]) $ we have 
          \begin{equation*} %\label{eq:Calabi-for-c0kB}
            \lim_{k \rightarrow \infty} c_{k,B_k}^0 (H) = \int_0^1 \int_{S^2} H_t\omega \, dt .
          \end{equation*}
          \end{theorem}
          \begin{proof}
          For deriving the Calabi property from Proposition \ref{prop:Calabi-general}, it is enough to verify the Locality property for the sequence $ c_{k,B_k}^0 $ of functionals. Choose a sequence 
          $ B_k' \in (B_k, \frac{1}{2}) $ that converges to $ 0 $ and such that 
          $ \lim_{k \rightarrow \infty} B_k' / B_k = \infty $. Then for the sequence $ \cD_m = z^{-1}([-1/2,-1/2+B_m')) $ of spherical discs, $(\ref{eq:equiv-formulation-locality})$ holds for 
          $ c_m := c_{m,B_m}^0 $ by the Lagrangian control property.
          \end{proof}
          
%\end{enumerate}
%\end{remark}

\subsection{The quotient $ \FHomeo(S^2) / \Hameo(S^2) $} \label{section:The-quotient}

%As pointed out previously, one can try to understand how large the quotient group $ \cG = \FHomeo(S^2) / \Hameo(S^2) $ is. The group $ \cG $ is naturally endowed with the Hofer pseudo-norm. It might be interesting to investigate the geometric group structure. 

Recall that for any $ \phi \in \FHomeo(S^2) $, its norm $ \| \phi \|_\cH $ is defined as the minimal possible $ \liminf_{k \rightarrow \infty} \| \phi_k \| $, where $ (\phi_k) $ is a sequence in $ \Ham(S^2) $ that $ C^0 $ converges to $ \phi $. The quotient $ \FHomeo(S^2) / \Hameo(S^2) $ is naturally endowed with the induced pseudo-norm $ \| \cdot \|_\cH $.

Consider the normed abelian additive group $ (l^\infty, \| \cdot \|_\infty) $ that consists of infinite bounded sequences $ s = (s_1,s_2,\ldots) $ of real numbers, such that $ \| s \|_\infty = \sup |s_k | $, and where the group structure is standard. Moreover, consider the subgroup $ c_0 \subset l^\infty $ that consists of all sequences that converge to $ 0 $. The norm $ \| \cdot \|_\infty $ on $ l^\infty $ naturally descends to a norm on the quotient $ l^\infty / c_0 $. We denote that norm on $ l^\infty / c_0 $ also by $ \| \cdot \|_\infty $. Note that for every $ s \in l^\infty $ and the corresponding element $ [s] \in l^\infty / c_0 $ we have 
$$ \| [s] \|_\infty = \limsup_{k \rightarrow \infty} |s_k| .$$

Using properties of the functionals $ c_{k,B}^0 $ (in particular, the Calabi property stated in Theorem \ref{thm:Calabi-for-c0kB}), combined with a soft approach, we show (cf. \cite[Theorem A]{PS}):

\begin{theorem} \label{thm:quotient-emb-flats}
The normed group $ (l^\infty / c_0 ,  \| \cdot \|_\infty) $ embeds isometrically into the group $ \cG = \FHomeo(S^2) / \Hameo(S^2) $ endowed with the Hofer pseudo-norm $ \| \cdot \|_{\cH} $. 
\end{theorem}

\begin{corollary} \label{cor:quotient-emb-flats}
One can isometrically embed into $ \cG $ the normed group $ (l^\infty ,  \| \cdot \|_\infty) $, and also the normed group $ (C(X), \| \cdot \|_\infty) $ when $ X $ is a separable topological space (e.g. when $ X = \R $). 
\end{corollary}
\noindent The corollary readily follows from the theorem and from the fact that one can isometrically embed the normed group $ (l^\infty ,  \| \cdot \|_\infty) $ into the normed group $ (l^\infty / c_0 ,  \| \cdot \|_\infty) $, and moreover isometrically embed $ (C(X), \| \cdot \|_\infty) $ into $ (l^\infty ,  \| \cdot \|_\infty) $ (for a separable topological space $ X $). We also remark that Theorem \ref{thm:quotient-emb-flats} is stronger than Corollary \ref{cor:quotient-emb-flats} in the sense that there is no isometric group embedding of $ (l^\infty / c_0 ,  \| \cdot \|_\infty) $ into
$ (l^\infty ,  \| \cdot \|_\infty) $. See Section \ref{subsubsection:Some-properties} for more details.

It is known that the group $ \cG $ is abelian \cite[Proposition 2.2]{CHMSS}. Theorem \ref{thm:quotient-emb-flats} implies that the torsion-free rank of $ \cG $ is continuum (since the cardinality of $ \cG $ is continuum itself). Our approach, however, does not seem to help understanding the torsion part of $ \cG $. Also, it is would be interesting to verify whether the Hofer pseudo-norm on $ \cG $ is non-degenerate (that is, a genuine norm). 
In addition, the following questions remain unanswered. For a smooth function $ h : (-1/2,1/2) \rightarrow \R $, consider the time-$1$ map $ \phi $ of the Hamiltonian flow of the autonomous function $ H(x_1,x_2,x_3) = h(x_3) $ defined on the sphere without the north and south poles, and extend $ \phi $ to a homeomorphism of $ S^2 $. For which $ h $ do we have $ \phi \in \FHomeo(S^2) $, and in that case what is the norm $ \| \phi \|_\cH $ (if not precisely, then up to a quasi-equivalence)? For which $ h $ do we have $ \phi \in \Hameo(S^2) $? Our proof of Theorem \ref{thm:quotient-emb-flats} shows only a very partial answer to these questions.

To conclude, we remark that the proofs of Theorems \ref{thm:main} and \ref{thm:quotient-emb-flats}, and of Proposition \ref{prop:Calabi-general}, are quite close. Morally speaking, Theorem \ref{thm:quotient-emb-flats} is more general than Theorem \ref{thm:main}. We still preferred to keep Theorem \ref{thm:main} for the convenience of the reader, since it serves a good motivation for the latter, and since its proof is simpler than that of Theorem \ref{thm:quotient-emb-flats} and does not rely on the Calabi property for the functionals $ c_{k,B}^0 $.

\subsection{Notation and preliminary remarks} \label{section:Notation}
%Given a smooth manifold $ (M,\omega) $, $ \Homeo(M) $ denotes the group of compactly supported homeomorphisms of $ M $. The $ C^0 $ convergence in $ \Homeo(M) $ is always uniformly compactly supported.

Let $ (M,\omega) $ be a symplectic surface. For a subset $ A \subset M $, its $ \omega $-area is denoted by $ \omega(A) $.
%By $ \| \cdot \| $ we denote the $ L_\infty $ norm on spaces $ C^\infty(M \times [0,1]) $ and $ C^\infty(M) $.
For a Hamiltonian function $ H \in C^\infty(M \times [0,1]) $, $ (\phi_H^t)_{t \in [0,1]} $ denotes the Hamiltonian flow of $ H $. For given Hamiltonian functions $ H, K \in C^\infty(M \times [0,1]) $, we denote $ H \sharp K (x,t) = H(x,t) + K((\phi_H^t)^{-1}(x)) $, and that Hamiltonian function generates the composition flow $ (\phi_H^t \circ \phi_K^t) $. Moreover by $ \overline{H}(x,t) = -H(\phi_H^t(x),t) $ we denote Hamiltonian function that generates the inverse flow $ ((\phi_H^t)^{-1}) $. The $ L^\infty $ norm of a Hamiltonian $ H \in C^\infty(M \times [0,1]) $ is denoted by $ \| H \| = \max |H| $. We say that a Hamiltonian $ H \in C^\infty(M \times [0,1]) $ is normalized if $ \int_M H_t \omega = 0 $ for all $ t \in [0,1] $. 

Now let $ (M,\omega) $ be a closed symplectic surface. The Hofer distance on $ \Ham(M,\omega) $ is denoted by $ d_\H $, and for every $ \phi \in \Ham(M,\omega) $, $ \| \phi \|_\H $ stands for the Hofer norm of $ \phi $. We also denote by $ d_\H $ and $ \| \cdot \|_\H $ the Hofer distance and norm on  $ \widetilde{\Ham}(M,\omega) $ (the fundamental cover of $ \Ham(M,\omega) $). In order to speak about $ C^0 $ convergence, we equip $ M $ with an auxiliary Riemannian metric which defines the distance function $ d : M \times M \rightarrow \R $. For homeomorphisms $ \phi, \psi : M \rightarrow M $, we define $ d_{C^0}(\phi,\psi) = \max_{x \in M} d(\phi(x),\psi(x)) $ and $ \od_{C^0}(\phi,\psi) = \max(d_{C^0}(\phi,\psi),d_{C^0}(\phi^{-1},\psi^{-1})) $.
The reason for considering the metric $ \od_{C^0} $ is related to the following important property: $ \Homeo(M) $ is complete with respect to $ \od_{C^0} $, that is,
if we have a Cauchy sequence in $ \Homeo(M) $ with respect to $ \od_{C^0} $ then it necessarily converges (with respect to $ \od_{C^0} $) to a homeomorphism of $ M $. For a sequence of homeomorphisms of $ M $, in order to conclude its uniform convergence to some homeomorphism, the Cauchy property with respect to $ d_{C^0} $ is generally not enough. Note however that if a sequence in $ \Homeo(M) $ is known to converge uniformly (i.e. with respect to $ d_{C^0} $) a homeomorphism, then it in fact converges to that homeomorphism with respect to $ \od_{C^0} $.
For $ \phi \in \Homeo(M) $ and a sequence $ \phi_k \in \Homeo(M) $ we will write $ \phi = (C^0)\lim_{k\rightarrow \infty} \phi_k $ if $ \lim_{k \rightarrow \infty} d_{C^0}(\phi,\phi_k) = 0 $ (which is equivalent to $ \lim_{k \rightarrow \infty} \od_{C^0}(\phi,\phi_k) = 0 $).

%$ \FHomeo(M,\omega) $ was first introduced in the breakthrough paper \lev{CITE}. The \lev{consequent(?)} paper \cite{CHS-2} defines Hofer's metric on $ \FHomeo(M,\omega) $. \lev{The definition is as follows: ... $ \tilde{d}_H(\phi, \psi) $}. We will call it Hofer-Le Roux metric to avoid confusion (this name is inspired by the question of Le Roux \lev{CITE} such that positive answer on it would imply that on $ \Ham(M,\omega) $ the metrics $ d_H $ and $ \tilde d_H $ coincide). 
%
%We consider $ (S^2,\omega) $, and $ \omega $ is normalized such that $ \omega(S^2) = 1 $, and think of it as sitting inside $ \R^3 $ as the sphere of radius $ \frac{1}{2} $ centered at
%the origin, equipped with the standard area form divided by $ \pi $. We denote by $ z : S^2 \rightarrow \mathbb{R} $ the $z$-coordinate function.

\subsection{Acknowledgements}

I thank Leonid Polterovich and Egor Shelukhin for an interest in this work and for fruitful discussions. I thank Vincent Humili\`ere and Sobhan Seyfaddini for sharing the question on the distinctness between the groups $ \Hameo $ and $ \FHomeo $, and for their comments on a preliminary version of this work. The author was partially supported by ERC Starting Grant 757585 and ISF Grant 2026/17.

\section{Proofs} \label{section:Proofs}

%Before passing to the proofs, we make the following remark:
%
%\begin{remark}
%The Normalization and Independence of Hamiltonian properties of $ c_{k,B}^0 $ imply that for any $ H \in C^\infty(S^2 \times [0,1]) $ we have 
%$$ c_{k,B}^0(H) = c_{k,B}^0(\phi_H^1) + \int_0^1 \int_{S^2} H_t\omega \, dt .$$
%Consequently, the Subadditivity property can be reformulated into the statement that for $ G,H \in C^\infty(M \times [0,1]) $ we have 
%\begin{equation} \label{eq:Subadditivity-reformulation}
%  c_{k,B}^0(G \sharp H) \leqslant c_{k,B}^0(G) + c_{k,B}^0(H) .
%\end{equation}
%\end{remark}

The central lemma which is used in the proofs is essentially due to Sikorav \cite[Section 8.4]{Si}:

\begin{lemma} \label{lemma:Sikorav-style}
Let $ (M,\omega) $ be a closed and connected symplectic surface. Let $ \eps > 0 $, let $ m $ be a positive integer, and let $ \cD_0, \ldots, \cD_m \subset M $ be topological open discs of area $ \eps $ each, such that for every $ 1 \leqslant j \leqslant m $ we are given a symplectic diffeomorphism $ \phi_j : \cD_0 \rightarrow \cD_j $. Moreover, let $ f_0, f_1, \ldots, f_m \in \Ham(M,\omega) $ with $ \supp (f_j) \subset \cD_j $. Define $ \Phi, \Phi' \in \Ham(M,\omega) $ by 
$$ \Phi = f_0 f_1 \cdots f_m $$ and $$ \Phi' = f_0 \Pi_{j=1}^m \phi_j^* f_j  ,$$ where $ \phi_j^* f_j $ is given by $ \phi_j^* f_j  = (\phi_j)^{-1} f_j \phi_j $ on $ \cD_0 $ and $ \phi_j^* f_j = \Id $ on $ M \setminus \cD_0 $. Then 
\begin{equation} \label{eq:statement-Sikorav-style}
 d_\H(\Phi,\Phi') < 3 \eps .
\end{equation}
\end{lemma}
\noindent Its proof will be given in Section \ref{subsection:Sikorav-style-lemma-proof} below. 

\begin{remark} \label{remark:Sikorav-style}
As it can be seen from the proof of Lemma \ref{lemma:Sikorav-style}, if one removes from the statement of the lemma the assumption that $ M $ is closed, one still gets
a similar conclusion that $ \Phi^{-1} \Phi' $ is generated by a normalized Hamiltonian function $ H $ of $ L^{(1,\infty)} $ norm less than $ 3\eps $. Moreover, by a reparametrization we can have $ \| H \| < 3\eps $.
\end{remark}

Proofs of Theorems \ref{thm:main} and \ref{thm:quotient-emb-flats} use the notation that we now introduce. For each integer $ m \geqslant 3 $ make a choice of real numbers $ \frac{1}{m+1} < B_m < B_m' < \frac{1}{m} $ and denote $ C_m = (1-2B_m)/(m-1) $ and $ C_m' = (1-2B_m')/(m-1) $. Then denote $ \sigma_m := c_{m,B_m}^0 $ and $ \sigma_m' := c_{m,B_m'}^0 $, and let $ \tau_{m} := \sigma_m - \sigma_m' $.

The main ingredient in the proof of Theorem \ref{thm:main} is the following lemma:

\begin{lemma} \label{lemma:main}
Let $ H: S^2 \times [0,1] \rightarrow \R $ be a smooth normalized Hamiltonian function. Then for every $ E > \| H \| $, $ \eps > 0 $ and every smooth Hamiltonian function $ F : S^2 \times [0,1] \rightarrow \R $, one can find an integer $ m \geqslant 3 $ and a smooth normalized Hamiltonian function $ H' : S^2 \times [0,1] \rightarrow \R $ such that:
\renewcommand\labelenumi{(\theenumi)}
\begin{enumerate}
 \item $ \od_{C^0} (\phi_H^t,\phi_{H'}^t) < \eps $ for every $ t \in [0,1] $.
 \item $ \| H' \| < E$.
 \item $ |\tau_m (H' \sharp F)| > 2E - \eps $.
\end{enumerate}
\end{lemma} 
\noindent  Its proof relies on Lemma \ref{lemma:Sikorav-style}, see Section \ref{subsection:lemma-main-proof}. 

\subsection{Proof of Theorem \ref{thm:main}}
%\begin{proof}[Proof of Theorem \ref{thm:main}]

Let $ E > 0 $ and let $ H: S^2 \times [0,1] \rightarrow \R $ be a smooth normalized Hamiltonian function with $ \| H \|_{(1,\infty)} \leqslant E $. We need to show that for every $ \eps > 0 $ we can find a continuous path $ (\varphi^t)_{t\in[0,1]} $ in $ \Homeo(S^2) $ such that $ d_{C^0}(\varphi^t,\phi_H^t) < \eps $ for all $ t \in [0,1] $, and which satisfies the properties $(1)$ and $(2)$ stated in the theorem. 
By approximation and reparametrization we may without loss of generality assume that $ \| H \| < E $.

The space of smooth normalized functions $ S^2 \times [0,1] \rightarrow \R $ is separable while endowed with the $ L^{(1,\infty)} $ norm. Choose a corresponding dense sequence of normalized Hamiltonian functions $ F_j : S^2 \times [0,1] \rightarrow \R $, $ (j=1,2,\ldots$), and then the sequence of time-$1$ maps $ \Phi_j := \phi_{F_j}^1 $ is dense in $ \Ham(S^2) $ with respect to the Hofer metric. We now inductively construct a sequence $ H_0, H_1, \ldots $ of time dependent normalized Hamiltonian functions on $ S^2 $. Set $ H_0 = H $, and for each $ k \geqslant 1 $, Lemma \ref{lemma:main} provides us a smooth Hamiltonian function $ H_k : S^2 \times [0,1] \rightarrow \R $ such that $ \od_{C^0} (\phi_{H_{k-1}}^t,\phi_{H_k}^t) < \eps/2^{k} $ for every $ t \in [0,1] $, such that $ \| H_k \| < E$, and such that for some $ m_k \geqslant 3 $ we have $ |\tau_{m_k}(H_k \sharp \overline{F_k})| > 2E - 1/k $. Moreover by the $ C^0 $-continuity of $ \tau_{m_j} $'s as functionals on $ \Ham(S^2) $, we may inductively assume that we also have $ |\tau_{m_j}(H_k \sharp \overline{F_j})| > 2E - 1 / j $ for every $ 1 \leqslant j < k $ (indeed, on the step $ k \geqslant 1 $ when we obtain $ H_k $, we get the bound $ |\tau_{m_j}(H_{k-1} \sharp \overline{F_j})| > 2E - 1 / j $ from the previous step, and then the inequality $ |\tau_{m_j}(H_{k} \sharp \overline{F_j})| > 2E - 1 / j $ follows from the $ C^0 $ continuity of $ \tau_{m_j} $ provided that $ d_{C^0}(\phi_{H_k}^1,\phi_{H_{k-1}}^1) $ is small enough -- and the item $(1)$ from the statement of Lemma \ref{lemma:main} guarantees that smallness).

To summarise, we now have a Hofer dense in $ \Ham(S^2) $ sequence $ \Phi_1,\Phi_2,\ldots $, where $ \Phi_j = \phi_{F_j}^1 $, and we have a sequence of smooth normalized Hamiltonian functions $ (H_k) $ and a sequence of indices $ (m_k) $, such that:
\begin{itemize}
\item[(a)] $ H_0 = H $.
\item[(b)] $ \od_{C^0} (\phi_{H_{k-1}}^t,\phi_{H_k}^t) < \eps/2^{k} $.
\item[(c)]  $ \| H_k \| < E $.
\item[(d)] $ |\tau_{m_j}(H_k \sharp \overline{F_j})| > 2E - 1 / j $ for $ 1 \leqslant j \leqslant k $.
\end{itemize}
The property (b) implies that the sequence of flows $ (\phi_{H_k}^t) $ uniformly converges to a continuous flow $ \varphi^t $ of homeomorphisms of $ S^2 $. Then by (c) we readily get the property $(1)$ from the statement of the theorem. Also note that (b) yields 
$$ \od_{C^0}(\phi_H^t,\varphi^t) = \od_{C^0}(\phi_{H_0}^t,\varphi^t) < \eps, \,\,\,\,\, \forall t \in [0,1]. $$

To see the property $ (2) $ we use the fact that $ \tau_{m_j} $ are Hofer Lipschitz with constant $ 2 $. As in $(2)$, assume that we have a sequence $ \psi_k \in \Ham(S^2) $ such that $ \varphi^1 = (C^0) \lim_{k\rightarrow\infty} \psi_k $, and let $ \psi \in \Ham(S^2) $. Given any $ \delta > 0 $ there exist infinitely many indices $ j $ such that $ d_\H(\psi,\Phi_j) < \delta $. Then for such $ j $ and for every $ k $ have 
\begin{equation} \label{eq:Hofer-tau-ineq}
2d_\H(\psi_k,\psi) \geqslant 2d_\H(\psi_k,\Phi_j) - 2 \delta \geqslant  |\tau_{m_j}(\psi_k \circ \Phi_j^{-1})| - 2\delta .
\end{equation}
Since $ \tau_{m_j} $ is $ C^0 $ continuous on $ \Ham(S^2) $ and moreover extends by continuity to $ \OHam(S^2) $, we get 
\begin{equation} \label{eq:tau-eq}
\lim_{k\rightarrow \infty} \tau_{m_j}(\psi_k \circ \Phi_j^{-1}) = \tau_{m_j}(\varphi^1 \circ \Phi_j^{-1}) = \lim_{k\rightarrow \infty} \tau_{m_j}(\phi_{H_k}^1 \circ \Phi_j^{-1}) .
\end{equation}
But then by $(\ref{eq:Hofer-tau-ineq})$, $(\ref{eq:tau-eq})$ and the property (d) we get
\begin{equation*}
\begin{gathered}
2 \liminf_{k \rightarrow \infty} d_\H(\psi_k,\psi) \geqslant  \lim_{k\rightarrow \infty} |\tau_{m_j}(\psi_k \circ \Phi_j^{-1})| - 2\delta \\
= \lim_{k\rightarrow \infty} |\tau_{m_j}(\phi_{H_k}^1 \circ \Phi_j^{-1})| - 2\delta = \lim_{k\rightarrow \infty} |\tau_{m_j}(H_k \sharp \overline{F_j})| - 2\delta \\
\geqslant 2E - 1 / j - 2 \delta
\end{gathered}
\end{equation*} 
We conclude that the inequality 
$$ \liminf_{k \rightarrow \infty} d_\H(\psi_k,\psi) \geqslant  E - 1 / j - \delta $$
holds for every $ \delta > 0 $ and for infinitely many integer values of $ j $. The property (2) now follows.
%\end{proof}

%\begin{lemma} \label{lemma:main}
%Let $ H: S^2 \times [0,1] \rightarrow \R $ be a smooth Hamiltonian function. Then for every $ E \geqslant \| H \| $, $ \eps > 0 $ and every smooth Hamiltonian function $ F : S^2 \times [0,1] \rightarrow \R $, one can find an integer $ m \geqslant 3 $ and a smooth Hamiltonian function $ H' : S^2 \times [0,1] \rightarrow \R $ such that:
%\renewcommand\labelenumi{(\theenumi)}
%\begin{enumerate}
% \item $ \od_{C^0} (\phi_H^t,\phi_{H'}^t) < \eps $ for every $ t \in [0,1] $.
% \item $ \| H' \| \leqslant E$.
% \item $ |\tau_m (H' \sharp F)| > 2E - \eps $.
%\end{enumerate}
%\end{lemma} 

\subsection{Proof of Lemma \ref{lemma:main}} \label{subsection:lemma-main-proof}
%\begin{proof}[Proof of Lemma \ref{lemma:main}]

Recall\footnote{See the beginning of Section \ref{section:Proofs}, right after the formulation of Lemma \ref{lemma:Sikorav-style}.} that for each integer $ m \geqslant 3 $ we have chosen real numbers $ \frac{1}{m+1} < B_m < B_m' < \frac{1}{m} $, and denoted $ C_m = (1-2B_m)/(m-1) $ and $ C_m' = (1-2B_m')/(m-1) $. Also, we denoted $ \sigma_m = c_{m,B_m}^0 $, $ \sigma_m' = c_{m,B_m'}^0 $, and $ \tau_{m} = \sigma_m - \sigma_m' $. 

Let $ m > 6/\eps $ be large enough. One can cover $ S^2 $ by open topological discs $ \cD_j $ ($0 \leqslant j < m$) of equal area $ A_m \in (B_m', \frac{1}{m}) $, such that the their diameters are bounded from above by $ o(1) $ when $ m \rightarrow \infty $ (the discs can in fact be chosen such that the diameters are bounded by $ m^{-1/2} $ up to a constant). We may without loss of generality assume that $ \cD_0 = (x_3)^{-1}([-1/2, -1/2 + A_m)) $. Denote $ L = (x_3)^{-1}(-1/2+B_m) $ and $ L' = (x_3)^{-1}(-1/2+B_m') $, where $ x_3  : S^2 \rightarrow \R $ is the coordinate function as before. We have $ L, L' \subset \cD_0 \subset S^2 $. Choose area-preserving diffeomorphisms $ \phi_j : \cD_0 \rightarrow \cD_j $, for $ 0 \leqslant j < m $ (where $ \phi_0 $ is taken to be the identity map). 

Define $ \hat{a}_j(t) = \inf_{x \in \cD_j} H(x,t) $ and $ \hat{b}_j(t) = \sup_{x \in \cD_j} H(x,t) $, ($ t \in [0,1] $). These functions are continuous on $ [0,1] $ and attain their values in $ (-E,E) $. Moreover if we chose $ m $ to be sufficiently large, we get $ 0 \leqslant \hat{b}_j(t) - \hat{a}_j(t) < \eps/2 $ for every $ j $ and $ t $. Then we can approximate $ \hat{a}_j , \hat{b}_j $ by smooth functions $ a_j, b_j : [0,1] \rightarrow (-E,E) $ such that we have $ \hat{a}_j > a_j $ and $ \hat{b}_j < b_j $ on $ [0,1] $, and such that we still have $ b_j(t) - a_j(t) < \eps/2 $.
% Just take e.g. a_j to be the convolution of $ \min(\hat{a}_j+epsilon,1) $ with a molifier

Denoting $ \delta = \frac{1}{4E}(E - \| H \|) > 0 $, pick smooth functions $ h,h' : S^2 \rightarrow [-\delta,1] $ which are compactly supported in $ \cD_0 $, have disjoint supports, such that 
$$ \int_{S^2} h\omega = \int_{S^2} h'\omega = 0 ,$$
and such that $ h|_{L} \equiv 1 $, $ h|_{L'} \equiv 0 $, $ h'|_{L} \equiv 0 $, $ h'|_{L'} \equiv 1 $. For each $ j $, define smooth functions $ h_j, h_j' : S^2 \rightarrow [-\delta,1] $, compactly supported in $ \cD_j $, as push-forwards $ h_j = (\phi_j)_* h $ and $ h_j' = (\phi_j)_* h' $. 

Now define normalized Hamiltonian functions $ K,K' : S^2 \times [0,1] \rightarrow \R $ by $$ K(x,t) = \sum_{j=0}^{m-1} (E-b_j(t)) h_j(x) + (-E-a_j(t)) h_j'(x) ,$$ $$ K'(x,t) = \sum_{j=0}^{m-1} (-E-a_j(t)) h_j(x) + (E-b_j(t)) h_j'(x) .$$ Clearly, $ K $ and $ K' $ have commuting Hamiltonian flows, and are both compactly supported in the disjoint union $ \cup_{j=0}^{m-1} \cD_j $. Define Hamiltonian functions $ H_1, H_2 : S^2 \times [0,1] \rightarrow \R $ by $ H_1 = K \sharp H $ and $ H_2 = K' \sharp H $. Let us show that we can choose the desired Hamiltonian $ H' $ to be either $ H_1 $ or $ H_2 $.

Introduce new Hamiltonian functions $$ K_0(x,t) = \sum_{j=0}^{m-1} (E-b_j(t)) h(x) + (-E-a_j(t)) h'(x) ,$$ $$ K_0'(x,t) = \sum_{j=0}^{m-1} (-E-a_j(t)) h(x) + (E-b_j(t)) h'(x) ,$$ which
are compactly supported in $ \cD_0 $. Lemma \ref{lemma:Sikorav-style} ensures that 
\begin{equation} \label{eq:Hofer-estimate}
\begin{gathered}
d_\H(\phi_{K'}^1  (\phi_{K}^1)^{-1}, \phi_{K_0'}^1 (\phi_{K_0}^1)^{-1}) < 3 A_m < 3/m < \eps/2 .
\end{gathered}
\end{equation}

%Also note that 
%\begin{equation*} 
%\begin{gathered}
%\int_0^1 \int_{S^2} K(x,t) \omega \, dt = \\
%\sum_{j=0}^{m-1} \left( E - \int_0^1 b_j(t) \, dt \right) \left( \int_{S^2} h_j(x) \omega \right)  + \sum_{j=0}^{m-1} \left( -E - \int_0^1 a_j(t) dt \right)  \left( \int_{S^2} h_j'(x) \omega \right) \\
%= \sum_{j=0}^{m-1} \left( E - \int_0^1 b_j(t) \, dt \right) \left( \int_{S^2} h(x) \omega \right)  + \sum_{j=0}^{m-1} \left( -E - \int_0^1 a_j(t) dt \right)  \left( \int_{S^2} h'(x) \omega \right) \\
%= \int_0^1 \int_{S^2} K_0(x,t) \omega \, dt .
%\end{gathered}
%\end{equation*}
%Similarly we conclude $$ \int_0^1 \int_{S^2} K'(x,t) \omega \, dt = \int_0^1 \int_{S^2} K_0'(x,t) \omega \, dt . $$
%Therefore
%\begin{equation} \label{eq:equality-of-Calabi}
%\begin{gathered}
%\int_0^1 \int_{S^2} K' \sharp \overline{K} (x,t) \omega \, dt = \int_0^1 \int_{S^2} K'(x,t) \omega \, dt - \int_0^1 \int_{S^2} K(x,t) \omega \, dt \\
%= \int_0^1 \int_{S^2} K_0'(x,t) \omega \, dt - \int_0^1 \int_{S^2} K_0(x,t) \omega \, dt = \int_0^1 \int_{S^2} K_0' \sharp \overline{K_0} (x,t) \omega \, dt .
%\end{gathered}
%\end{equation}

Recall our notation $ \sigma_m = c_{m,B_m}^0 $ and $ \sigma_m' = c_{m,B_m'}^0 $, and $ \tau_{m} = \sigma_m - \sigma_m' $. The inequality $(\ref{eq:Hofer-estimate})$, and the Hofer Lipschitz and Independence of Hamiltonian properties of $ \sigma_m $ yield
\begin{equation*} %\label{eq:sigma-Hofer-estimate}
\begin{gathered}
 \sigma_m(\phi_{K'}^1 \circ (\phi_{K}^1)^{-1}) \leqslant \sigma_m(\phi_{K_0'}^1 \circ (\phi_{K_0}^1)^{-1}) + d_\H(\phi_{K'}^1  (\phi_{K}^1)^{-1}, \phi_{K_0'}^1 (\phi_{K_0}^1)^{-1}) \\
 <  \sigma_m(\phi_{K_0'}^1 \circ (\phi_{K_0}^1)^{-1}) + \eps/2 
\end{gathered}
\end{equation*}
Now because $ H_1 = K \sharp H $ and $ H_2 = K' \sharp H $, we therefore get 
\begin{equation} \label{eq:sigma-m-estimate-1}
\begin{gathered}
\sigma_m(H_2 \sharp \overline{H_1})  = \sigma_m(K' \sharp H \sharp \overline{H} \sharp \overline{K}) = \sigma_m(K' \sharp \overline{K}) \\
= \sigma_m(\phi_{K'}^1 \circ (\phi_{K}^1)^{-1})   < \sigma_m(\phi_{K_0'}^1 \circ (\phi_{K_0}^1)^{-1}) + \eps/2  = \sigma_m(K_0' \sharp \overline{K_0})  + \eps /2  .
\end{gathered}
\end{equation}
The value $ \sigma_m(K_0' \sharp \overline{K_0}) $ can be computed explicitly. Indeed, notice that 
\begin{equation*}
\begin{gathered}
 K_0' \sharp \overline{K_0} (x,t) = K_0'(x,t) - K_0 (x,t) \\ 
 = \left( -2mE + \sum_{j=0}^{m-1} (b_j(t) - a_j(t)) \right)h(x) +  \left( 2mE + \sum_{j=0}^{m-1} (a_j(t) - b_j(t)) \right)h'(x) .
\end{gathered}
\end{equation*}
Hence by our choice of the functions $ h $ and $ h' $, we see that for each $ t \in [0,1] $ we have $ K_0' \sharp \overline{K_0} (x,t) =  -2m + \sum_{j=0}^{m-1} (b_j(t) - a_j(t)) $ when 
$ x \in L = z^{-1}(-1/2+B_m) $. Moreover we have $ K_0' \sharp \overline{K_0} (x,t) = 0 $ when $ x \in z^{-1}(-1/2+B_m + i C_m) $ for $ i = 1, \ldots, m-1 $. To see the this, note that $$ B_m + C_m = B_m + (1-2B_m)/(m-1) = ((m-3)B_m + 1)/(m-1) > 1/m > A_m ,$$ and so the circles $ z^{-1}(-1/2+B_m + i C_m) $ lie in the complement of the supports of $ h $ and $ h' $ ($ i = 1, \ldots, m-1 $). Therefore by the Lagrangian control property we get
\begin{equation} \label{eq:sigma-m-estimate-2}
\begin{gathered}
\sigma_m(K_0' \sharp \overline{K_0}) 
= \frac{1}{m} \int_0^1 \left( -2mE + \sum_{j=0}^{m-1} (b_j(t) - a_j(t)) \right) dt \leqslant -2E + \eps/2 .
\end{gathered}
\end{equation}
Combining this estimate with $(\ref{eq:sigma-m-estimate-1})$, we conclude 
\begin{equation} \label{eq:sigma-m-main-ineq}
\sigma_m(H_2 \sharp \overline{H_1}) < \sigma_m(K_0' \sharp \overline{K_0}) + \eps /2  \leqslant -2E + \eps .
\end{equation}
In an analogous way we obtain the lower bound 
\begin{equation}  \label{eq:sigma-m-tag-main-ineq}
 \sigma_m'(H_1 \sharp \overline{H_2}) < -2E+\eps .
\end{equation}
Indeed, similarly as in $(\ref{eq:sigma-m-estimate-1})$ and $(\ref{eq:sigma-m-estimate-2})$, we have estimates
\begin{equation*} %\label{eq:sigma-m-tag-estimate-1}
\begin{gathered}
 \sigma_m'(H_1 \sharp \overline{H_2}) < \sigma_m'(K_0 \sharp \overline{K_0'}) + \eps/2,
\end{gathered}
\end{equation*}
and 
\begin{equation*} %\label{eq:sigma-m-tag-estimate-2}
\sigma_m'( K_0 \sharp \overline{K_0'}) = \frac{1}{m} \int_0^1 \left( -2mE + \sum_{j=0}^{m-1} (b_j(t) - a_j(t)) \right) dt \leqslant -2E + \eps/2,
\end{equation*}
and the inequality $(\ref{eq:sigma-m-tag-main-ineq})$ follows.

By the triangle inequality, $(\ref{eq:sigma-m-main-ineq})$ and $(\ref{eq:sigma-m-tag-main-ineq})$ we have
\begin{equation*}
\begin{gathered}
\tau_m(H_2 \sharp F) = \sigma_m(H_2 \sharp F) - \sigma_m'(H_2 \sharp F) \\
\leqslant (\sigma_m(H_2 \sharp \overline{H_1}) + \sigma_m(H_1 \sharp F)) + (-\sigma_m'(H_1 \sharp F) + \sigma_m'(H_1 \sharp \overline{H_2})) \\
= \tau_m(H_1 \sharp F) + \sigma_m(H_2 \sharp \overline{H_1}) +  \sigma_m'(H_1 \sharp \overline{H_2}) < \tau_m(H_1 \sharp F) - 4E + 2\eps.
\end{gathered}
\end{equation*}
This means that we have either $ \tau_m(H_1 \sharp F) > 2E - \eps $ or $ \tau_m(H_2 \sharp F) < -2E + \eps $. In the first case we put $ H' = H_1 $, and in the second $ H' = H_2 $.
Then $ H' $ satisfies the property (3) from the statement of the lemma. To see the property (2), it is enough to check that $ \| H_1 \| < E $ and $ \| H_2 \| < E $. We have $ H_1(x,t) = K \sharp H (x,t) = K(x,t) + H((\Phi_K^t)^{-1}(x),t) $. Therefore on each $ \cD_j $, since the Hamiltonian flow of $ K $ preserves it, we have 
$$ H_1(x,t) \leqslant \max(E-b_j(t),2E\delta) + \sup_{\cD_j} H(\cdot,t) < E $$ and
$$ -H_1(x,t) \leqslant \max(E+a_j(t),2E\delta) - \inf_{\cD_j} H(\cdot,t) < E $$
(recall that $ a_j(t), b_j(t) \in (-E,E) $, $ a_j(t) < \hat{a}_j(t) = \inf_{\cD_j} H(\cdot,t) $, and  
$ b_j(t) > \hat{b}_j(t) = \sup_{\cD_j} H(\cdot,t) $). Moreover, on the complement of the union $ \cup_j \cD_j $ we have $ H_1(x,t) = H(x,t) $, and so we obtain $ |H_1(x,t)| < E $ as well. We conclude that $ \| H_1 \| < E $. The inequality $ \| H_2 \| < E $ follows similarly. Finally, the property (1) holds if $ m $ is sufficiently large (the Hamiltonian flows of $ K $ and $ K' $ are compactly supported in the disjoint union $ \cup_j \cD_j $, and the diameters of all the $ \cD_j $'s are bounded by $ o(1) $ when $ m \rightarrow \infty $).  
%\end{proof}

\subsection{Proof of Proposition \ref{prop:Calabi-general}}

%\begin{proof}[Proof of Proposition \ref{prop:Calabi-general}]
Note that the Normalization and Independence of Hamiltonian properties of $ c_m $ imply that the Subadditivity property can be reformulated into the statement that for $ G,H \in C^\infty(M \times [0,1]) $ we have 
\begin{equation} \label{eq:Subadditivity-reformulation}
  c_{m}(G \sharp H) \leqslant c_{m}(G) + c_{m}(H) .
\end{equation}
Then the Monotonicity and Normalization properties imply that for every smooth $ h : M \rightarrow (-\infty,0] $ and $ t \geqslant 0 $ we have 
\begin{equation} \label{eq:c-m-negative-sublinearity-with-error}
c_m(th) \leqslant (t-1)c(h). 
\end{equation}
Indeed, from the Monotinicity and Normalization properties we get $ c_m(h) \leqslant c({\bf{0}}) = 0 $, and then denoting $ n = \floor{t} $, from $(\ref{eq:Subadditivity-reformulation}) $ and Monotonicity we conclude $ c_m(th) \leqslant c_m(nh) \leqslant n c_m(h) \leqslant (t-1) c(h) $. 
%Also note that from the Locality property for $ c_m $ we in particular get  
%\begin{equation} \label{eq:c-m-h-m-converges-to-0}
%\lim_{m \rightarrow \infty} c_m(-h_m) = 0. 
%\end{equation}

%Now let $ H \in C^\infty(M \times [0,1]) $, and let $ \eps > 0 $ be arbitrary. Choose some $ E \geqslant \| H \| $. 

Denote $ N_m = \ceil{\omega(M) / \omega(\cD_m)} - 2 $ and note that
\begin{equation} \label{eq:N-m-c-m-lim}
\begin{gathered}
\lim_{m \rightarrow \infty} N_m = + \infty, \\ 
\lim_{m \rightarrow \infty} N_m c_m(-h_m) = - 1 .
\end{gathered}
\end{equation}
We can without loss of generality assume that the open disc $ \cD_m $ has a smooth boundary. Pack $ M $ by open topological discs $ \cD_{m,j} $ ($ 0 \leqslant j \leqslant N_m$) of equal areas, that have smooth boundaries and mutually disjoint closures, such that $ \cD_{m,0} = \cD_m $, and such that the diameters of $ \cD_{m,0}, \ldots, \cD_{m,N_m} $ are bounded from above by $ o(1) $ when $ m \rightarrow \infty $. Choose area-preserving diffeomorphisms $ \phi_{m,j} : \cD_{m,0} \rightarrow \cD_{m,j} $, for $ 0 \leqslant j \leqslant N_m $ (where $ \phi_{m,0} $ is taken to be the identity map). For each $ 0 \leqslant j \leqslant N_m $, define smooth functions $ h_{m,j}: S^2 \rightarrow [0,1] $, compactly supported in $ \cD_{m,j} $, as push-forwards $ h_{m,j} = (\phi_{m,j})_* h_m $. 

Now let $ H \in C^\infty(M \times [0,1]) $, and let $ \eps > 0 $ be arbitrary. Choose some $ E > \| H \| $. 
Then, choose $ m $ to be large enough, such that in particular we have $ 3\omega(\cD_m) < \eps $ and $ \omega(\cD_m) \| H \| < \eps $.
%Now let $ H \in C^\infty(M \times [0,1]) $. Choose some $ E \geqslant \| H \| $ and set 
Define
$$ \hat{a}_{j}(t) = \inf_{\cD_{m,j}} H(\cdot,t) ,$$
$$ \hat{b}_{j}(t) = \sup_{\cD_{m,j}} H(\cdot,t) ,$$  
($ t \in [0,1] $). These functions are continuous on $ [0,1] $ and attain their values in $ (-E,E) $. Moreover if we choose $ m $ to be sufficiently large, we get $ 0 \leqslant \hat{b}_{j}(t) - \hat{a}_{j}(t) < \eps $ for every $ j $ and $ t $. Then approximate $ \hat{a}_j , \hat{b}_j $ by smooth functions $ a_j, b_j  : [0,1] \rightarrow (-E,E) $ such that we have $ \hat{a}_j > a_j $ and $ \hat{b}_j < b_j $ on $ [0,1] $, and such that we still have $ b_j(t) - a_j(t) < \eps $. Denote 
$$ I = \int_0^1 \int_{M} H(x,t) \omega\, dt .$$ Then $ I $ is well--approximated by $ \omega(\cD_m) \left(\sum_{j=0}^{N_m} \int_0^1 b_j(t) \, dt \right) $. Indeed, note that the complement $ K := M \setminus \cup_{j=0}^{N_m} \cD_{m,j} $ of the union of the discs, has a small area:
$$ \omega(K) = \omega(M) - (N_m + 1) \omega(\cD_m) \leqslant \omega(\cD_m) .$$
Moreover,  
$$ I = \int_0^1 \int_{M} H(x,t) \omega\, dt = \left( \sum_{j=0}^{N_m} \int_0^1 \int_{\cD_{m,j}} H(x,t) \omega\, dt \right) + \int_0^1 \int_{K} H(x,t) \omega\, dt .$$
We have 
$$  \left| \int_0^1 \int_{K} H(x,t) \omega\, dt \right| \leqslant \omega(K) \| H \| \leqslant \omega(\cD_m) \| H \| < \eps. $$
Also, by 
$$ a_j(t) < \inf_{\cD_{m,j}} H(\cdot ,t) \leqslant \sup_{\cD_{m,j}} H(\cdot ,t) < b_j(t) $$ 
and by $ b_j(t) - a_j(t) < \eps $, we have 
$$ | H(x,t) - b_j(t) | < \eps $$ 
on $ \cD_{m,j} $, and therefore 
$$ \left| \int_0^1 \int_{\cD_{m,j}} H(x,t) \omega\, dt  - \omega(\cD_m) \int_0^1 b_j(t) \, dt \right| < \eps \omega(\cD_m)  $$
for each $ j $. We conclude 
\begin{equation} \label{eq:I-approx-c}
 \left| I - \omega(\cD_m) \left(\sum_{j=0}^{N_m} \int_0^1 b_j(t) \, dt \right) \right| < (1+ \omega(M)) \eps .
\end{equation}

Define smooth Hamiltonian functions $ K,K',K_0,K_0' : S^2 \times [0,1] \rightarrow \R $ by 
$$ K(x,t) = \sum_{j=0}^{N_m} (E-b_j(t)) h_{m,j}(x) , \,\,\,\,\,\, K'(x,t) = \sum_{j=0}^{N_m} (-E-a_j(t)) h_{m,j}(x) ,$$ 
$$ K_0(x,t) = \sum_{j=0}^{N_m} (E-b_j(t)) h_{m}(x) = \left( (N_m + 1)E - \sum_{j=0}^{N_m} b_j(t) \right) h_m(x) ,$$ 
$$ K_0'(x,t) = \sum_{j=0}^{N_m} (-E-a_j(t)) h_{m}(x) = \left( -(N_m + 1)E - \sum_{j=0}^{N_m} a_j(t) \right) h_m(x).$$ 
%$ K $ and $ K' $ have commuting Hamiltonian flows, and are both compactly supported in the disjoint union $ \cup_{j=0}^{N_m} \cD_{m,j} $. 
Similarly as in the proof of Lemma \ref{lemma:main}, applying the version of Lemma \ref{lemma:Sikorav-style} for the universal cover $ \Ham(M,\omega) $ as described in Section \ref{subsubsection:Sikorav-for-universal-cover}, we conclude that 
$$ d_\H([(\phi_{K_0}^t)],[(\phi_{K}^t)]) \, , \, d_\H([(\phi_{K_0'}^t)],[(\phi_{K'}^t)]) \, < \, 3\omega(\cD_m) < \eps  . $$
It is easy to see that we have
$$ \int_0^1 \int_M K(x,t) \omega \, dt = \int_0^1 \int_M K_0(x,t) \omega \, dt $$
and 
$$ \int_0^1 \int_M K'(x,t) \omega \, dt = \int_0^1 \int_M K_0'(x,t) \omega \, dt .$$
Hence by the Hofer Lipschitz and Independence of Hamiltonian properties of $ c_m $ we get 
\begin{equation*} %\label{eq:HL-IH-c-m}
| c_m(K) - c_m(K_0)|, \, |c_m(K')-c_m(K_0')| < \eps.  
\end{equation*}
Also, by $(\ref{eq:c-m-negative-sublinearity-with-error})$, and by the Normalization and Independence of Hamiltonian properties we have
$$ c_m(\overline{K_0}) = c_m(-K_0) \leqslant \left( (N_m + 1)E - 1 - \sum_{j=0}^{N_m} \int_0^1 b_j(t) \, dt \right) c_m(-h_m) ,$$
$$ c_m(K_0') \leqslant  \left( (N_m + 1)E - 1 + \sum_{j=0}^{N_m} \int_0^1 a_j(t) \, dt \right) c_m(-h_m) .$$
Therefore we conclude 
$$ c_m(\overline{K}) < \left( (N_m + 1)E - 1 - \sum_{j=0}^{N_m} \int_0^1 b_j(t) \, dt \right) c_m(-h_m) + \eps =: A_m ,$$
$$ c_m(K') <  \left( (N_m + 1)E - 1 + \sum_{j=0}^{N_m} \int_0^1 a_j(t) \, dt \right) c_m(-h_m) + \eps =: B_m .$$
%In particular, we get $ c_m(K' \sharp \overline{K}) \leqslant  c_m(K') + c_m(\overline{K}) < A_m + B_m $. 
Note that
\begin{equation*}
\begin{gathered}
 A_m + B_m = \left( 2(N_m + 1)E - 2 + \sum_{j=0}^{N_m} \int_0^1 (a_j(t)-b_j(t)) \, dt \right) c_m(-h_m) + 2 \eps \\
 < (2N_m E + 2E - 2 - (N_m + 1) \eps ) c_m(-h_m) + 2 \eps 
\end{gathered}
\end{equation*}
Therefore, if $ m $ is large enough then in view of $(\ref{eq:N-m-c-m-lim})$ we get $ A_m + B_m < -2E+ 4\eps $.
Define Hamiltonian functions $ H_1, H_2 : S^2 \times [0,1] \rightarrow \R $ by $ H_1 = K \sharp H $ and $ H_2 = K' \sharp H $. Then it is easy to see that we have $ \| H_1 \|, \| H_2 \| \leqslant E $.
In particular we get $ -E \leqslant c_m(H_1), c_m (H_2) \leqslant E $. Moreover, we have $ c_m(H) \leqslant c_m(H_1) + c_m (\overline{K}) < c_m(H_1) + A_m $ and 
$ c_m(H_2) \leqslant c_m(H) + c_m (K') < c_m(H) + B_m $. 

Let us summarize what we have:
\begin{itemize} 
 \item $A_m + B_m < -2E + 4\eps$,
 \item $-E \leqslant c_m(H_1), c_m (H_2) \leqslant E$,
 \item $c_m(H) < c_m(H_1) + A_m$,
 \item $c_m(H_2) < c_m(H) + B_m$.
\end{itemize}
This implies that $ c_m(H) $ lies in the interval $ (-E-B_m,E+A_m) $ whose length is $ (E+A_m)-(-E-B_m) = 2E + A_m + B_m < 4\eps $. In particular, 
$$ |c_m(H) - (E+A_m)| < 4\eps .$$
We have 
\begin{equation*}
\begin{gathered}
E + A_m =  E + \left( (N_m + 1)E - 1 - \sum_{j=0}^{N_m} \int_0^1 b_j(t) \, dt \right) c_m(-h_m) + \eps \\
= (N_m c_m(-h_m) + 1) E + (E-1)c_m(-h_m) + \\
 + \left( I - \omega(\cD_m) \sum_{j=0}^{N_m} \int_0^1 b_j(t) \, dt \right) (c_m(-h_m)/\omega(\cD_m)) -  \\
 - ((\omega(M))^{-1} + c_m(-h_m)/\omega(\cD_m)) I + \eps + (\omega(M))^{-1} I .
\end{gathered}
\end{equation*}
Hence by $(\ref{eq:N-m-c-m-lim})$, $(\ref{eq:I-approx-c})$ and the Locality property we have $$ | E + A_m - (\omega(M))^{-1}I | < (3+(\omega(M))^{-1}) \eps, $$
provided that $ m $ is large enough. Therefore we get $$ |c_m(H) - (\omega(M))^{-1}I | < (7+(\omega(M))^{-1}) \eps $$ for large enough $ m $. Since $ \eps > 0 $ is arbitrary, this shows $(\ref{eq:Calabi-property})$.
%\end{proof}

\subsection{Proof of Theorem \ref{thm:quotient-emb-flats}} 

%
%{\lev{ STATEMENT: The normed group $ (l^\infty,  \| \cdot \|_\infty) $ embeds isometrically into the group $ \cG = \FHomeo(S^2) / \Hameo(S^2) $ endowed with the Hofer pseudo-metric.} \\ \\

Denote by $ \pi: \FHomeo(S^2) \rightarrow \FHomeo(S^2) / \Hameo(S^2) $ the natural projection homomorphism. For proving the theorem it is enough to find a group homomorphism $ \Phi : l^\infty \rightarrow \FHomeo(S^2) $, such that for each $ s \in l^\infty $ we have
\begin{equation} \label{eq:quotient-emb-flats}
 \| \pi \circ \Phi (s) \|_\cH = \limsup_{k \rightarrow \infty} |s_k| ,
\end{equation}
and such that $ \Phi(c_0) \subset \Hameo(S^2) $. Indeed, then $ \pi \circ \Phi $ naturally descends to an isometric embedding $ l^\infty / c_0 \rightarrow \FHomeo(S^2) / \Hameo(S^2) $.
%In the proof, we will embed $ l^\infty $ into $ \cG $, and will show that the embedding is isometric. The latter 
After constructing $ \Phi $, the proof of the equality $(\ref{eq:quotient-emb-flats})$ will be divided into proving the inequalities ``$\geqslant$'' and ``$\leqslant$''.
%showing that for $ s \in l^\infty $, the Hofer norm of its image is not less and not greater than the norm of $ s $. 
For the first inequality we will use properties of the functionals $ c_{k,B}^0 $, in particular the Calabi property stated in Theorem \ref{thm:Calabi-for-c0kB}. For the opposite one, we will apply a soft argument which uses Lemma \ref{lemma:Sikorav-style}.

Recall\footnote{See the beginning of Section \ref{section:Proofs}, right after the formulation of Lemma \ref{lemma:Sikorav-style}.} that for each integer $ m \geqslant 3 $ we have chosen real numbers $ \frac{1}{m+1} < B_m < B_m' < \frac{1}{m} $, and denoted $ C_m = (1-2B_m)/(m-1) $ and $ C_m' = (1-2B_m')/(m-1) $. Also, we denoted $ \sigma_m = c_{m,B_m}^0 $, $ \sigma_m' = c_{m,B_m'}^0 $, and $ \tau_{m} = \sigma_m - \sigma_m' $. For each $ m \geqslant 3 $ choose some $ A_m \in (B_m',\frac{1}{m}) $ and consider the open disc $$ \cD_m := (x_3)^{-1}([-1/2,-1/2+A_m)).$$ 
Let 
$$ \cA_m := (x_3)^{-1}((1/(m+1),A_m)) \subset \cD_m \subset S^2 $$ 
denote the spherical annulus, and let $ L_m := (x_3)^{-1}(B_m) $ and $ L_m' := (x_3)^{-1}(B_m') $ be circles inside $ \cA_m $. 

%For each integer $ m \geqslant 3 $ choose real numbers $ \frac{1}{m+1} < B_m < B_m' < \frac{1}{m} $ and denote $ C_m = (1-2B_m)/(m-1) $ and $ C_m' = (1-2B_m')/(m-1) $. Further, denote $ \sigma_m := c_{B_m,C_m}^0 $ and $ \sigma_m' := c_{B_m',C_m'}^0 $, and let $ \tau_{m} := \sigma_m - \sigma_m' $. Let $ \cA_m := z^{-1}((1/(m+1),1/m)) \subset S^2 $ denote the spherical annulus, and let $ L_m := z^{-1}(B_m) $ and $ L_m' := z^{-1}(B_m') $ circles in it. 
%
%Let $ m \geqslant 3 $ be an integer. Consider the open interval $ I_m := \left( \frac{1}{m+1} , \frac{1}{m} \right) $, and then choose some $ B_m \in I_m $ and denote $ C_m = (1-2B_m)/(m-1) $. Further, denote $ \sigma_m := c_{B_m,C_m}^0 $, and let $ \cA_m := z^{-1}(I_m) \subset S^2 $ be the spherical annulus. Also, consider the open disc $ \cD_m := z^{-1}([-1/2,-1/2+B_m)) \supset \cA_m $.

For each $ m \geqslant 3 $, pack $ M $ by open topological discs $ \cD_{m,j} $ ($ 0 \leqslant j \leqslant m-1$) of equal areas, that have smooth boundaries and mutually disjoint closures, such that $ \cD_{m,0} = \cD_m $, and such that the diameters of $ \cD_{m,0}, \ldots, \cD_{m,m-1} $ are bounded from above by $ o(1) $ when $ m \rightarrow \infty $. Choose area-preserving diffeomorphisms $ \phi_{m,j} : \cD_{m,0} \rightarrow \cD_{m,j} $, for $ 0 \leqslant j \leqslant m-1 $ (where $ \phi_{m,0} $ is taken to be the identity map).

In order to prepare for the construction, we inductively choose a sequence of indices $ m_1 < m_2 < \ldots $, and also choose a sequence of smooth functions $ h_k : S^2 \rightarrow \R $, such that each $ h_k $ depends only on the $ x_3 $-coordinate and is supported in $ \cA_{m_k} $. Moreover, some auxiliary functions $ F_{k,s} $ will be introduced, and later they will be useful in estimating the Hofer norm.

At the first step we choose $ m_1 \geqslant 3 $ arbitrarily, and then choose smooth functions $ h_1', h_1'': S^2 \rightarrow [0,1] $ that depend only on the coordinate $ x_3 $, such that their disjoint supports are contained in the annulus $ \mathcal{A}_{m_1} $, such that $ h_1'|_{L_{m_1}} \equiv 1 $, $ h_1'|_{L_{m_1}'} \equiv 0 $, $ h_1''|_{L_{m_1}} \equiv 0 $, $ h_1''|_{L_{m_1}'} \equiv 1 $, and $ \int_{S^2} h_1' \omega = \int_{S^2} h_1'' \omega $. Then define $ h_1 : S^2 \rightarrow \R $ by $ h_1 := h_1'-h_1'' $, and for each $ s \in [-1,1] $ define $ F_{1,s} \in C^\infty(S^2 \times [0,1]) $
as $ F_{1,s}(x,t) = sh_1(x) $. This finishes the first step.

Now we describe a step $ k $ when $ k > 1 $. In the previous step we have constructed the family $ (F_{k-1,s} (\cdot,\cdot))_{s \in [-1,1]^{k-1}} $ of functions in $ C^\infty(S^2 \times [0,1]) $. By the construction and by compactness considerations, there exists a constant $ C_{k-1} $ such that for every $ t \in [0,1] $ and $ s \in [-1,1]^{k-1} $, the function $ F_{k-1,s}(\cdot,t) : S^2 \rightarrow \R $ is 
$ (C_{k-1} \max_{1 \leqslant i \leqslant k-1} |s_i| ) $--Lipschitz. Also, by the construction, for each $ s \in [-1,1]^{k-1} $ the Hamiltonian $ F_{k-1,s} $ in normalized, that is  
\begin{equation} \label{eq:nrmlzd-F} 
\int_{S^2} F_{k-1,s}(x,t) \omega  = 0 
\end{equation} 
for every $ t \in [0,1] $. Moreover we have 
\begin{equation} \label{eq:Fks-estimate}
 \| F_{k-1,s} \| \leqslant \max_{1 \leqslant i \leqslant k-1} (1+ 2^{-i}-2^{-k}) |s_i| .
\end{equation}
Hence we can choose $ m_{k} > 2 m_{k-1} $ sufficiently large such that for every $ s \in [-1,1]^{k-1} $, $ t \in [0,1] $ and $ j $ we have 
\begin{equation} \label{eq:osc-F}
\underset{\cD_{m_{k},j}}{\osc} F_{k-1,s}(\cdot,t) =  \sup_{\cD_{m_{k},j}} F_{k-1,s}(\cdot,t) - \inf_{\cD_{m_{k},j}} F_{k-1,s}(\cdot,t) \leqslant \frac{\|s\|_\infty}{2^{k+6}} ,
\end{equation}
and also denoting by $ Z_k $ the complement of the union $ \cup_{j=0}^{m_k-1} \cD_{m_k,j} $, we have 
\begin{equation} \label{eq:F-int-rest}
 \int_{Z_k} |F_{k-1,s}(x,t)| \omega \leqslant \frac{\|s\|_\infty}{2^{k+5}} 
\end{equation}
for every $ t \in [0,1] $. Moreover, we may assume that $ m_k $ is sufficiently large so that $ | m_i \tau_{m_k}(h_i) | \leqslant 2^{-k} $ for every $ 1 \leqslant i \leqslant k-1 $ (since $ h_k $ depends only on the $ x_3 $-coordinate, for verifying 
that condition it is enough to refer only to the Lagrangian control property).

For each $ s \in [-1,1]^{k-1} $ and $ j $ choose a smooth function $ a_{k,j}( \cdot ; s) : [0,1]  \rightarrow \R $, such that\footnote{The dependence of $ a_{k,j}(t;s) $ on $ s \in [-1,1]^{k-1} $ is not required to be smooth or even continuous. An important property of $ a_{k,j} $ is the boundedness which is uniform in $ s $, and it follows from $(\ref{eq:choice-of-akj})$.}
\begin{equation} \label{eq:choice-of-akj}
\inf_{\cD_{m_{k},j}} F_{k-1,s}(\cdot,t) - \frac{\|s\|_\infty}{2^{k+6}} \leqslant a_{k,j}(t; s)  \leqslant  \sup_{\cD_{m_{k},j}} F_{k-1,s}(\cdot,t) + \frac{\|s\|_\infty}{2^{k+6}} 
\end{equation}
By $(\ref{eq:nrmlzd-F})$ and $(\ref{eq:F-int-rest})$ we have 
\begin{equation*}
\left| \sum_{j=0}^{m_k-1} \int_{\cD_{m_k,j}} F_{k-1,s}(x,t) \omega \right| \leqslant \frac{\|s\|_\infty}{2^{k+5}}, 
\end{equation*} 
hence by $(\ref{eq:osc-F})$ and $(\ref{eq:choice-of-akj})$ we get
\begin{equation*}
\omega(\cD_{m_k}) \left| \sum_{j=0}^{m_{k}-1} a_{k,j} (t;s) \right| \leqslant \frac{\|s\|_\infty}{2^{k+4}},
\end{equation*} 
for every $ t \in [0,1] $. Therefore the average
\begin{equation*}
%\begin{gathered}
A_k(t;s):= \frac{1}{m_k} \sum_{j=0}^{m_{k}-1} a_{k,j} (t;s)  
%\end{gathered}
\end{equation*}
verifies
\begin{equation} \label{eq:Ak-average-estimate}
\begin{gathered}
|A_k(t;s)| \leqslant \frac{\|s\|_\infty}{2^{k+4}m_k\omega(\cD_{m_k})} \leqslant \frac{\|s\|_\infty}{2^{k+3}} .
\end{gathered}
\end{equation}
Denoting $ \tilde a_{k,j}(t;s) := a_{k,j}(t;s) - A_k(t;s) $, we have
\begin{equation} \label{eq:avrg-akj}
\sum_{j=0}^{m_k-1} \tilde a_{k,j}(t;s) = 0
\end{equation}
for every $ t $ and $ s $. Choose smooth functions $ h_k', h_k'': S^2 \rightarrow [0,1] $ that depend only on the coordinate $ x_3 $, such that their disjoint supports are contained in the annulus $ \mathcal{A}_{m_k} $, such that $ h_k'|_{L_{m_k}} \equiv 1 $, $ h_k'|_{L_{m_k}'} \equiv 0 $, $ h_k''|_{L_{m_k}} \equiv 0 $, $ h_k''|_{L_{m_k}'} \equiv 1 $, and $ \int_{S^2} h_k' \omega = \int_{S^2} h_k'' \omega $. 

Now let $ s = (s_1,\ldots,s_k) \in [-1,1]^k $, and denote $ s'= (s_1,\ldots,s_{k-1}) \in [-1,1]^{k-1} $. If $ s_k = 0 $ then we set $ F_{k,s} := F_{k-1,s'} $. Otherwise, define the Hamiltonian function 
$$ G_{k,s}(x,t) = \sum_{j=0}^{m_{k}-1} \left( (s_k - \tilde a_{k,j} (t;s'))  (\phi_{m_k,j})_* h_k'(x) + (-s_k - \tilde a_{k,j} (t;s'))  (\phi_{m_k,j})_* h_k''(x) \right) ,$$
and then set
$$ F_{k,s} := G_{k,s} \sharp F_{k-1,s'} . $$
By $(\ref{eq:avrg-akj})$ and by the choice of $ h_k' $ and $ h_k'' $ we get that $ G_{k,s} $ is normalized, and hence $ F_{k,s} $ is normalized as well.

Consider the case when $ s_k \neq 0 $. For every $ (x,t) \in S^2 \times [0,1] $ we have $ F_{k,s}(x,t) = G_{k,s}(x,t) + F_{k-1,s'}((\phi_{G_{k,s}}^t)^{-1}(x),t) $. Hence for $ x \notin \cup_j \cD_{m_k,j} $ we have $ F_{k,s}(x,t) = F_{k-1,s'}(x,t) $. Otherwise, if $ x \in \cD_{m_k,j} $ for some $ j $, then in the case when $ x $ lies outside the supports of $ (\phi_{m_k,j})_* h_k' $ and $ (\phi_{m_k,j})_* h_k'' $, we get the same conclusion $ F_{k,s}(x,t) = F_{k-1,s'}(x,t) $. It remains to treat separately the cases of $ x \in \supp((\phi_{m_k,j})_* h_k') $ and $ x \in \supp((\phi_{m_k,j})_* h_k'') $, where our aim is to estimate $ F_{k,s}(x,t) $. Assuming $ x \in \supp((\phi_{m_k,j})_* h_k') $, we get $ x \notin \supp((\phi_{m_k,j})_* h_k'') $, and hence for a given $ t \in [0,1] $, denoting $ \lambda = (\phi_{m_k,j})_* h_k' (x) \in [0,1] $ and $ y = (\phi_{G_{k,s}}^t)^{-1}(x) \in \cD_{m_k,j} $, we have 
\begin{equation*}
\begin{gathered}
F_{k,s}(x,t) = \lambda (s_k - \tilde a_{k,j} (t;s'))  + F_{k-1,s'}(y,t) \\
= \lambda (s_k - a_{k,j} (t;s') + A_k(t;s'))  + F_{k-1,s'}(y,t) \\ 
= (1- \lambda)F_{k-1,s'}(y,t) + \lambda s_k + \lambda (F_{k-1,s'}(y,t) - a_{k,j} (t;s') + A_k(t;s')). 
\end{gathered}
\end{equation*}
But by $(\ref{eq:choice-of-akj})$ and $(\ref{eq:Ak-average-estimate})$ we have 
$$ |F_{k-1,s'}(y,t) - a_{k,j} (t;s') + A_k(t;s')| \leqslant \frac{\|s\|_\infty}{2^{k+2}}. $$
Hence by $(\ref{eq:Fks-estimate})$ we conclude
\begin{equation*}
\begin{gathered}
 | F_{k,s}(x,t) | \leqslant \max_{1 \leqslant i \leqslant k} (1+ 2^{-i}-2^{-k}) |s_i| + 2^{-k-2} \|s\|_\infty \\
 \leqslant \max_{1 \leqslant i \leqslant k} (1+ 2^{-k-2})(1+ 2^{-i}-2^{-k}) |s_i| \leqslant \max_{1 \leqslant i \leqslant k} (1+ 2^{-i} - 2^{-k-1}) |s_i|.
\end{gathered}
\end{equation*}
We obtained this estimate under the assumption of $ x \in \supp((\phi_{m_k,j})_* h_k') $. The case of $ x \in \supp((\phi_{m_k,j})_* h_k'') $ is completely analogous, and the same conclusion follows.
Thus we get
\begin{equation} \label{eq:Fks-Hofer-main-est}
\| F_{k,s} \| \leqslant \max_{1 \leqslant i \leqslant k} (1+ 2^{-i} - 2^{-k-1}) |s_i| . 
\end{equation}
Now if $ s_k = 0 $, then $(\ref{eq:Fks-Hofer-main-est})$ holds as well, by the inductive assumption. Finally, define the function $ h_k \in C^\infty(S^2) $ by $ h_k := h_k' - h_k'' $. This finishes the step $ k $.

Note that if we have $ s \in [-1,1]^k $ and $ s_k = 0 $, then denoting $ s':= (s_1,\ldots,s_{k-1}) \in [-1,1]^{k-1} $, by the construction we have $ F_{k,s} \equiv F_{k-1,s'} $. Hence for each $ s \in l^\infty $ which has only a finite number of non-zero coordinates, we can define $ F_s(x,t) := F_{\ell,\sigma}(x,t) $ where $ \sigma = (s_1,\ldots,s_\ell) \in [-1,1]^\ell $, and $ \ell $ is such that $ s_k = 0 $ for $ k > \ell $. 

For every $ s = (s_1,s_2, \ldots) \in l^\infty $ denote $ {}^ks:= (s_1,s_2, \ldots, s_k, 0,0, \ldots) \in l^\infty $ to be the element of $ l^\infty $ whose first $ k $ coordinates coincide with those of $ s $ and whose remaining coordinates are $ 0 $. Moreover, we define $ s^k := s - {}^{(k-1)}s = (0,\ldots,0,s_{k},s_{k+1},\ldots) \in l^\infty $. 

For each $ s \in l^\infty $ which has only a finite number of non-zero coordinates, define $ \Phi (s) \in \Ham(S^2) $ to be the time-$1$ map of the autonomous Hamiltonian $ \sum m_k s_k h_k $. 
Then, for every $ s \in l^\infty $ define 
$$ \Phi(s) := (C^0)\lim_{k \rightarrow \infty} \Phi ({}^ks) \in \overline{\Ham}(S^2) .$$
We claim that for each $ s \in l^\infty $ with $ \| s \|_\infty \leqslant 1 $ we in fact have $ \Phi(s) \in \FHomeo(S^2) $ and moreover 
\begin{equation} \label{eq:Phi-s-Hofer-estimate}
\| \Phi(s) \|_{\cH} \leqslant \sup_{k} \, (1+ 2^{-k}) |s_k| + 3 \sum_{s_k \neq 0} \frac{1}{m_k} .
\end{equation}
For proving this, it is enough to show that for every $ s \in l^\infty $, $ \| s \|_\infty \leqslant 1 $, such that only a finite number of its coordinates are non-zero, we have
\begin{equation} \label{eq:Phi-s-00-Hofer-estimate}
\| \Phi(s) \|_{\H} \leqslant \max_{k} \, (1+ 2^{-k}) |s_k| + 3 \sum_{s_k \neq 0} \frac{1}{m_k} .
\end{equation}
Indeed, having $(\ref{eq:Phi-s-00-Hofer-estimate})$, for any $ s \in l^\infty $, $ \| s \|_\infty \leqslant 1 $, we get
$$ \Phi(s) = (C^0)\lim_{\ell \rightarrow \infty} \Phi ({}^\ell s) ,$$
and
$$ \| \Phi({}^\ell s) \|_{\H} \leqslant \sup_{k} \, (1+ 2^{-k}) |s_k| + 3 \sum_{s_k \neq 0} \frac{1}{m_k}  $$
for each $ \ell \geqslant 1 $, hence $ \Phi(s) \in \FHomeo(S^2) $, and $(\ref{eq:Phi-s-Hofer-estimate})$ follows just by the definition of the norm $ \| \cdot \|_\cH $ on $ \FHomeo(S^2) $.

To show $(\ref{eq:Phi-s-00-Hofer-estimate})$, it is enough to prove that 
\begin{equation} \label{eq:Phi-s-00-Hofer-estimate-auxil}
d_\H(\Phi(s), \phi_{F_{s}}^1) \leqslant 3 \sum_{s_k \neq 0} \frac{1}{m_k} .
\end{equation}
Indeed, then $(\ref{eq:Phi-s-00-Hofer-estimate})$ will readily follow from $(\ref{eq:Phi-s-00-Hofer-estimate-auxil})$ and $(\ref{eq:Fks-Hofer-main-est})$. Now, let us prove $(\ref{eq:Phi-s-00-Hofer-estimate-auxil})$ by induction in $ \ell:= \max \left( \{1\} \cup \{ k \, | \, s_k \neq 0 \} \right) $. In the case of $ \ell = 1 $ we have $ \Phi(s) = \phi_{F_{s}}^1 $ and hence $(\ref{eq:Phi-s-00-Hofer-estimate-auxil})$ clearly holds. If $ \ell > 1 $ then denoting 
$ \sigma := (s_1, \ldots, s_{\ell}) \in [-1,1]^{\ell} $, $ \sigma' := (s_1, \ldots, s_{\ell-1}) \in [-1,1]^{\ell-1} $ and $ s' := {}^{(\ell-1)}s = (s_1, \ldots, s_{\ell-1},0,0,\ldots) \in l^\infty $, from the induction hypothesis it follows that 
\begin{equation*} %\label{eq:Phi-s-00-Hofer-estimate-auxil-induct}
d_\H(\Phi(s'), \phi_{F_{s'}}^1) \leqslant 3 \sum_{s_k \neq 0 \, , \, k < \ell} \frac{1}{m_k} .
\end{equation*}
In addition we have 
$$ \Phi(s) = \phi_{m_\ell s_\ell h_\ell}^1 \Phi(s') $$
and 
$$  \phi_{F_{s}}^1 = \phi_{F_{\ell,\sigma}}^1 = \phi_{G_{\ell,\sigma}}^1 \phi_{F_{\ell-1,\sigma'}}^1 = \phi_{G_{\ell,\sigma}}^1 \phi_{F_{s'}}^1 $$
where
$$ G_{\ell,\sigma}(x,t) = \sum_{j=0}^{m_{\ell}-1} \left( (s_\ell - \tilde a_{\ell,j} (t;\sigma'))  (\phi_{m_\ell,j})_* h_\ell'(x) + (-s_\ell - \tilde a_{\ell,j} (t;\sigma'))  (\phi_{m_\ell,j})_* h_\ell''(x) \right) $$
as before. By the construction we always have
$$ \sum_{j=0}^{m_{\ell}-1}  \tilde a_{\ell,j} (t;\sigma') = 0, $$
hence a direct application of Lemma \ref{lemma:Sikorav-style} yields
$$ d_\H(\phi_{G_{\ell,\sigma}}^1,\phi_{m_\ell s_\ell h_\ell}^1) < \frac{3}{m_\ell} .$$
Hence we conclude $(\ref{eq:Phi-s-00-Hofer-estimate-auxil})$ by the triangle inequality.

We have proved $(\ref{eq:Phi-s-Hofer-estimate})$. Now as a corollary we get that for every $ s \in l^\infty $ we have $ \Phi(s) \in \FHomeo(S^2) $ and
\begin{equation} \label{eq:Phi-s-Hofer-estimate-unrestrict}
\| \Phi(s) \|_{\cH} \leqslant \sup_{k} (1+ 2^{-k}) |s_k| + 3 (\| s \|_\infty + 1) \sum_{s_k \neq 0} \frac{1}{m_k} .
\end{equation}
Indeed, if $ s \in l^\infty $ then let $ N $ to be equal to the integer part of $ \| s \|_\infty $, and denote $ \tilde s := s/(N+1) $. Then clearly $ \| \tilde s \|_\infty \leqslant 1 $, 
hence $ \Phi(\tilde s) \in \FHomeo(S^2) $ and so $ \Phi(s) = (\Phi(\tilde s))^{N+1} \in \FHomeo(S^2) $, and moreover by $(\ref{eq:Phi-s-Hofer-estimate})$ we conclude
\begin{equation*}
\begin{gathered}
\| \Phi(s) \|_{\cH} \leqslant (N+1) \| \Phi(\tilde s) \|_{\cH} \leqslant (N+1) \sup_{k} (1+ 2^{-k}) |\tilde s_k| + 3 (N + 1) \sum_{\tilde s_k \neq 0} \frac{1}{m_k} \\
\leqslant \sup_{k} (1+ 2^{-k}) |s_k| + 3(\| s \|_\infty + 1) \sum_{s_k \neq 0} \frac{1}{m_k} .
\end{gathered}
\end{equation*}

Recall that we denoted by $ \pi : \FHomeo(S^2) \rightarrow \cG = \FHomeo(S^2) / \Hameo(S^2) $ the natural projection homomorphism. For every $ s \in l^\infty $ and every $ \ell \geqslant 1 $ we have
$$ \| \pi \circ \Phi (s) \|_\cH = \| \pi \circ \Phi (s^\ell) \|_\cH ,$$
and by $(\ref{eq:Phi-s-Hofer-estimate-unrestrict})$ we have
$$ \limsup_{\ell \rightarrow \infty} \| \pi \circ \Phi (s^\ell) \|_\cH \leqslant  \limsup_{\ell \rightarrow \infty} \| \Phi (s^\ell) \|_\cH \leqslant \limsup_{\ell \rightarrow \infty} | s_\ell | . $$
Thus we get the estimate 
\begin{equation} \label{eq:Phi-s-Hofer-estimate-main}
 \| \pi \circ \Phi (s) \|_\cH \leqslant \limsup_{\ell \rightarrow \infty} | s_\ell | 
\end{equation}
for every $ s \in l^\infty $. In order to conclude $(\ref{eq:quotient-emb-flats})$ it remains to show the opposite inequality, and for this we will apply properties of the functionals $ c_{k,B}^0 $.

One of important properties of the functionals $ \tau_m = \sigma_m - \sigma_m' = c_{m,B_m}^0 - c_{m,B_m'}^0 $ is that they are $ C^0 $-continuous on $ \Ham(S^2) $ and that they extend by continuity to $ \OHam(S^2) $. Consider any $ s \in l^\infty $ having a finite number of non-zero coordinates. Then by the Lagrangian control property for $ c_{m,B_m}^0 $ and $ c_{m,B_m'}^0 $ we get
\begin{equation*}
\begin{gathered}
\tau_{m_k} (\Phi(s)) = \tau_{m_k} \left( \sum_i m_i s_i h_i \right) = \sum_{i=1}^\infty m_i s_i \tau_{m_k}(h_i) =  m_k s_k \tau_{m_k}(h_k) + \sum_{i=1}^{k-1} m_i s_i \tau_{m_k}(h_i) 
\end{gathered}
\end{equation*}
Now, we have $ m_k \tau_{m_k}(h_k) = 2 $ and $ | m_i \tau_{m_k}(h_i) | \leqslant 2^{-k} $ for $ 1 \leqslant i \leqslant k-1 $, hence we get
\begin{equation*}
\begin{gathered}
\tau_{m_k} (\Phi(s)) \geqslant 2s_k - 2^{-k} k \| s \|_\infty .
\end{gathered}
\end{equation*}
Now, if $ \psi \in \Ham(M,\omega) $ is arbitrary, then 
\begin{equation*}
\begin{gathered}
\tau_{m_k} (\Phi(s) \psi^{-1}) = \sigma_{m_k}(\Phi(s) \psi^{-1}) - \sigma_{m_k}'(\Phi(s) \psi^{-1}) \\
\geqslant (\sigma_{m_k}(\Phi(s)) - \sigma_{m_k}(\psi)) - (\sigma_{m_k}'(\Phi(s)) + \sigma_{m_k}'(\psi^{-1})) \\
= \tau_{m_k}(\Phi(s)) - (\sigma_{m_k}(\psi) + \sigma_{m_k}'(\psi^{-1})) \geqslant 2s_k - 2^{-k} k \| s \|_\infty - (\sigma_{m_k}(\psi) + \sigma_{m_k}'(\psi^{-1})).
\end{gathered}
\end{equation*}
As a corollary, for any $ s \in l^\infty $, any $ \ell \geqslant k \geqslant 1 $, and every $ \psi \in \Ham(S^2) $ we get
\begin{equation*}
\tau_{m_k} (\Phi({}^\ell s) \psi^{-1}) \geqslant 2s_k - 2^{-k} k \| s \|_\infty - (\sigma_{m_k}(\psi) + \sigma_{m_k}'(\psi^{-1})).
\end{equation*}
But then the $ C^0 $ continuity of $ \tau_{m_k} $ implies 
\begin{equation} \label{eq:Psi-Hofer-prelim}
\tau_{m_k} (\Phi(s) \psi^{-1}) \geqslant 2s_k - 2^{-k} k \| s \|_\infty - (\sigma_{m_k}(\psi) + \sigma_{m_k}'(\psi^{-1})).
\end{equation}
Now let $ \varphi \in \Hameo(S^2) $, and choose a sequence of $ \psi_j \in \Ham(S^2) $ such that $ \varphi = (C^0)\lim_{j \rightarrow \infty} \psi_j $ and
$ d_\H(\psi_i,\psi_j) \leqslant 1/i $ for $ j \geqslant i \geqslant 1 $. Then by $ (\ref{eq:Psi-Hofer-prelim})$ and by the Hofer Lipschitz and Independence of Hamiltonian properties of $ \sigma_m $, $ \sigma_m' $, we get
\begin{equation*} 
\tau_{m_k} (\Phi(s) \psi_j^{-1}) \geqslant 2s_k - 2^{-k} k \| s \|_\infty - (\sigma_{m_k}(\psi_i) + \sigma_{m_k}'(\psi_i^{-1})) - 2/i.
\end{equation*}
for $ j \geqslant i \geqslant 1 $. Taking $ j \rightarrow \infty $, by the $ C^0 $ continuity of $ \tau_{m_k} $ we get
\begin{equation} \label{eq:Psi-Hofer-prelim-2}
\tau_{m_k} (\Phi(s) \varphi^{-1}) \geqslant 2s_k - 2^{-k} k \| s \|_\infty - (\sigma_{m_k}(\psi_i) + \sigma_{m_k}'(\psi_i^{-1})) - 2/i.
\end{equation}
But by the Calabi property stated in Theorem \ref{thm:Calabi-for-c0kB}, for each $ i \geqslant 1 $ we have
$$ \lim_{m \rightarrow \infty} \sigma_{m}(\psi_i) + \sigma_{m}'(\psi_i^{-1}) = \lim_{m \rightarrow \infty} c_{m,B_m}^0(\psi_i) + c_{m,B_m'}^0(\psi_i^{-1}) = 0 .$$
Hence from $(\ref{eq:Psi-Hofer-prelim-2})$, by taking $ k \rightarrow \infty $, we get
\begin{equation*} 
\limsup_{k \rightarrow \infty} \tau_{m_k} (\Phi(s) \varphi^{-1}) \geqslant 2 \limsup_{k \rightarrow \infty} s_k - 2/i
\end{equation*}
for every $ i \geqslant 1 $, and therefore we finally conclude that the inequality
\begin{equation} \label{eq:Psi-Hofer-prelim-3}
\limsup_{k \rightarrow \infty} \tau_{m_k} (\Phi(s) \varphi^{-1}) \geqslant 2 \limsup_{k \rightarrow \infty} s_k 
\end{equation}
holds for every $ s \in l^\infty $ and $ \varphi \in \Hameo(S^2) $. Now, we claim that given such $ s $ and $ \varphi $, we have
\begin{equation} \label{eq:Psi-Hofer-almost-done}
 \| \Phi(s) \varphi^{-1} \|_\cH \geqslant \limsup_{k \rightarrow \infty} s_k .
\end{equation}
Indeed, assume that we have a sequence $ \phi_i \in \Ham(S^2) $ that $ C^0 $ converges to $ \Phi(s) \varphi^{-1} $. Let $ \delta > 0 $ be arbitrary,
and then by $(\ref{eq:Psi-Hofer-prelim-3})$ we can find some $ k $ such that
$$ \tau_{m_k} (\Phi(s) \varphi^{-1}) \geqslant 2 \limsup_{k \rightarrow \infty} s_k - \delta .$$
But then by the $ C^0 $ continuity and by the Hofer Lipschitz property of $ \tau_{m_k} $ we get
$$ 2 \liminf_{i \rightarrow \infty} \| \phi_i \|_\H \geqslant \lim_{i \rightarrow \infty} \tau_{m_k} (\phi_i) = \tau_{m_k} (\Phi(s) \varphi^{-1}) \geqslant 2 \limsup_{k \rightarrow \infty} s_k - \delta ,$$
and since $ \delta > 0 $ is arbitrary, we conclude $(\ref{eq:Psi-Hofer-almost-done})$. It remains to notice that for a given $ s \in l^\infty $ and $ \varphi \in \Hameo(S^2) $, substituting $ -s $ and 
$ \varphi^{-1} $ into $(\ref{eq:Psi-Hofer-almost-done})$, by $ (\Phi(-s)\varphi)^{-1} = \varphi^{-1} \Phi(s) = \varphi^{-1} (\Phi(s) \varphi^{-1}) \varphi $ we get
\begin{equation*} 
\| \Phi(s) \varphi^{-1} \|_\cH = \| \Phi(-s) \varphi \|_\cH \geqslant \limsup_{k \rightarrow \infty} (-s_k) ,
\end{equation*}
and together with $(\ref{eq:Psi-Hofer-almost-done})$ this implies 
\begin{equation*} %\label{eq:Psi-Hofer-done}
 \| \Phi(s) \varphi^{-1} \|_\cH \geqslant \limsup_{k \rightarrow \infty} | s_k | 
\end{equation*}
for every $ s \in l^\infty $ and $ \varphi \in \Hameo(S^2) $, which means that we have
$$  \| \pi \circ \Phi(s) \|_\cH \geqslant \limsup_{k \rightarrow \infty} | s_k | $$
for every $ s \in l^\infty $, and this finishes the proof of $(\ref{eq:quotient-emb-flats})$.

It remains to show that the image $ \Phi(c_0) $ is contained in $ \Hameo(S^2) $. This will be done under additional assumptions on the sequence $ m_k $ and on functions $ h_k' $ and $ h_k'' $.
Namely, we require that on step $ k $ we choose $ m_k $ to be large enough, and choose a sufficiently thin neigbourhood $ W_k $ of $ L_{m_k} \cup L_{m_k}' $ in $ \cA_{m_k} $, such that if for each $ \ell $ and each $ \eta \in (0,1] $ we consider the collection $ \cJ(\ell,\eta) $ consisting of all the indices $ j $ such that the disc $ \cD_{m_\ell,j} $ lies inside $ S^2 \cap \{ x_3 < -1/2 + \eta \} $ but does not intersect the union $ \cup_{k=1}^{\ell-1} \cup_{j=0}^{m_k-1} \phi_{m_k,j}(W_k) $, then for every $ \eta \in (0,1] $ we have
\begin{equation} \label{eq:Phi-s-Hameo-part-1}
 | \cJ(\ell,\eta) | > m_\ell \eta /2 
\end{equation}
when $ \ell $ is large enough (recall that the area of the spherical cap $ S^2 \cap \{ x_3 < -1/2 + \eta \} $ equals $ \eta $, and that the area of a disc $ \cD_{m} $ behaves asymptotically like $ 1/m $). This requirement can be easily met by controlling on each step $ \ell $ the area of the intersection of the union $ \cup_{k=1}^{\ell-1} \cup_{j=0}^{m_k-1} \phi_{m_k,j}(W_k) $ with $ S^2 \cap \{ x_3 < -1/2 + \eta \} $, for every $ \eta \in (0,1] $ (it is enough to have that area to be less than $ \frac{1}{2} \omega(S^2 \cap \{ x_3 < -1/2 + \eta \}) = \eta/2 $). Once the sequences $ (m_k) $ and $ (W_k) $ are chosen, we choose the functions $ h_k' $ and $ h_k'' $ as before but impose an additional requirement that their supports lie in $ W_k $. As before, we put $ h_k = h_k' - h_k'' $.

Now assume that we have some $ s \in c_0 $, and let us show that $ \Phi(s) \in \Hameo(S^2) $. We can find a sequence $ m_k' \leqslant m_k $ of natural numbers such that $ \lim_{k \rightarrow \infty} m_k'/m_k = 0 $ and at the same time the sequence $ s_k' := m_ks_k/m_k' $ converges to $ 0 $ when $ k \rightarrow \infty $. By $(\ref{eq:Phi-s-Hameo-part-1})$, it is possible to choose a sequence $ \eta_k \in (0,1] $ and find an index $ k_0 $ such that 
$ \lim_{k \rightarrow \infty} \eta_k = 0 $, and such that $ m_k' \leqslant | \cJ(k,\eta_k) | $ for $ k \geqslant k_0 $. For every $ k \geqslant k_0 $ choose a subset $ 0 \in J_k \subset  \cJ(k,\eta_k) $ that has exactly $ m_k' $ elements.

For each $ k \geqslant k_0 $, consider the autonomous Hamiltonian function 
$$ G_k := s_k' \sum_{j \in J_k} (\phi_{m_k,j})_* h_k .$$
By Lemma \ref{lemma:Sikorav-style}, and Remark \ref{remark:Sikorav-style} applied to the surface $ S^2 \cap \{ x_3 < -1/2 + \eta_k \} $, one can find a normalized Hamiltonian function $ H_k \in C^\infty(S^2 \times [0,1]) $ compactly supported in $ S^2 \cap \{ x_3 < -1/2+\eta_k \} $, such that $ \phi_{m_k s_k h_k}^1 = \phi_{m_k' s_k' h_k}^1 = \phi_{G_{k}}^1 (\phi_{H_k}^1)^{-1} $, and such that 
\begin{equation} \label{eq:Hofer-bound-on-H-k}
 \| H_k \| < \frac{3}{m_k}  .
\end{equation}

Since the supports of the $ G_k $'s are mutually disjoint and are ``converging'' to the south pole of $ S^2 $, and since $ \lim_{k \rightarrow \infty} \| G_k \|_\infty = 0 $, it follows that 
$$ \Psi := (C^0) \lim_{k \rightarrow \infty} \phi_{G_k}^1 \circ \phi_{G_{k-1}}^1 \circ \cdots \circ \phi_{G_{k_0}}^1 $$
is an element of $ \Hameo(S^2) $ being the time-$1$ map of a continuous flow of homeomorphisms generated by the continuous (autonomous) Hamiltonian $ \sum_{k=k_0}^{\infty} G_k $. Therefore it is enough to prove that $ \Psi^{-1} \circ \Phi(s) \in \Hameo(S^2) $,
which is equivalent to $ (\Phi(s^{k_0}))^{-1} \circ \Psi  \in \Hameo(S^2) $.
We have 
$$ (\Phi(s^{k_0}))^{-1} \circ \Psi  = (C^0) \lim_{k \rightarrow \infty} \Phi_k^{-1} \Psi_k ,$$
where 
$$ \Psi_k = \phi_{G_k}^1 \circ \phi_{G_{k-1}}^1 \circ \cdots \circ \phi_{G_{k_0}}^1 $$
and 
$$ \Phi_k = \phi_{m_k s_k h_k}^1 \circ \cdots \circ \phi_{m_{k_0} s_{k_0} h_{k_0}}^1 $$
for $ k \geqslant k_0 $. Denoting $ \Phi_{k_0-1} := \Id $, for each $ k \geqslant k_0 $ we have
\begin{equation*}
\begin{gathered}
 \Phi_k^{-1} \Psi_k = ( \Phi_{k-1}^{-1} (\phi_{m_{k} s_{k} h_{k}}^1)^{-1} \phi_{G_{k}}^1 \Phi_{k-1} ) \circ \cdots \circ ( \Phi_{k_0-1}^{-1} (\phi_{m_{k_0} s_{k_0} h_{k_0}}^1)^{-1} 
\phi_{G_{k_0}}^1 \Phi_{k_0-1} ) \\
= ( \Phi_{k-1}^{-1} \phi_{H_k}^1 \Phi_{k-1} ) \circ \cdots \circ ( \Phi_{k_0-1}^{-1} \phi_{H_{k_0}}^1 \Phi_{k_0-1} ).
\end{gathered}
\end{equation*}

Choose a smooth non-decreasing function $ c : [0,1] \rightarrow [0, 1] $ such that $ c(t) = 0 $ for $ t $ close to $ 0 $ and $ c(t) = 1 $ for $ t $ close to $ 1 $, and for each $ j $ define 
$ c_j : [\frac{j-1}{j},\frac{j}{j+1}] \rightarrow \R $ by $ c_j(t) = c((j+1)(jt-j+1)) $. Now, for each $ k \geqslant k_0 $, define the Hamiltonian function $ \widetilde H_k \in C^\infty(S^2 \times [0,1]) $ by $ \widetilde H_k(x,t) = 0 $ when $ t \in [0,\frac{k_0-1}{k_0}] \cup [\frac{k}{k+1},1] $, and 
\begin{equation*}
\begin{gathered}
\widetilde H_k(x,t) = c_j'(t)H_j(\Phi_{j-1}(x),c_j(t)) \\
= j(j+1)c'((j+1)(jt-j+1)) H_j (\Phi_{j-1}(x), c((j+1)(jt-j+1))) 
\end{gathered}
\end{equation*}
for $ t \in [\frac{j-1}{j} ,\frac{j}{j+1}] $ and $ k_0 \leqslant j \leqslant k $. By $(\ref{eq:Hofer-bound-on-H-k})$ and by the rapid growth of the sequence $ (m_k) $ (we have chosen the sequence so that in particular $ m_k > 2 m_{k-1} $ for each $ k > 1 $), the sequence $ (\widetilde H_k) $ converges in $ \| \cdot \| $ to a continuous function $ \widetilde H : S^2 \times [0,1] \rightarrow \R $. Since the sequence $ (\widetilde H_k) $ of Hamiltonian functions stabilizes on $ S^2 \times [0,t] $, for every $ t \in [0,1) $, the function $ \widetilde H $ is smooth on $ S^2 \times [0,1) $, hence the time-$t$ map $ \phi_{\widetilde H}^t $ is well defined for $ t \in [0,1) $. We claim the convergence
\begin{equation} \label{eq:big-flow}
(C^0) \lim_{t \rightarrow 1} \phi_{\widetilde H}^t = (\Phi(s^{k_0}))^{-1} \circ \Psi .
\end{equation}
To see this, first note that we have that convergence along the sequence $ t_k = \frac{k}{k+1} $, since $ \phi_{\widetilde H}^{t_k} =  \Phi_k^{-1} \Psi_k $ for $ k \geqslant k_0 $.
Therefore $(\ref{eq:big-flow})$ follows from the fact that for each $ k \geqslant k_0 $, the flow 
$$ \phi_{\widetilde H}^{t_{k-1}+\tau} \circ (\phi_{\widetilde H}^{t_{k-1}})^{-1} = \Phi_{k-1}^{-1} \phi_{H_k}^{c_k(t_{k-1}+\tau)} \Phi_{k-1}, $$
($ \tau \in [0,\frac{1}{k(k+1)}] $) converges to the identity when $ k \rightarrow \infty $, uniformly in $ \tau $. The latter holds since the Hamiltonian $ H_k $ is supported in the spherical cap $ S^2 \cap \{ x_3 < -1/2 + \eta_k \} $ which is invariant under $ \Phi_{k-1} $, and we have $ \lim_{k \rightarrow \infty} \eta_k = 0 $ (that is, that spherical cap ``shrinks'' into the south pole of $ S^2 $ when $ k $ converges to infinity). 

Because of $(\ref{eq:big-flow})$, if we define $ \phi_{\widetilde H}^1 := (\Phi(s^{k_0}))^{-1} \circ \Psi $, then $ (\phi_{\widetilde H}^t)_{t \in [0,1]} $ is a continuous path of homeomorphisms. But 
then $(\ref{eq:big-flow})$ also implies that $ (\phi_{\widetilde H}^t)_{t \in [0,1]} $ is the $ C^0 $ limit of the sequence $ (\phi_{\widetilde H_k}^t)_{t \in [0,1]} $ of smooth Hamiltonian flows (this can be readily seen from the definition of $ \widetilde H_k $). Recall that we have also verified that the sequence $ (\widetilde H_k) $ of Hamiltonians converges in the $ L^\infty $ norm (in particular in the 
$ L^{(1,\infty)} $ norm) to $ \widetilde H $. This shows that $ (\Phi(s^{k_0}))^{-1} \circ \Psi = \phi_{\widetilde H}^1 \in \Hameo(S^2) $, and hence finishes the proof.

\subsection{Proof of Lemma \ref{lemma:Sikorav-style}} \label{subsection:Sikorav-style-lemma-proof}

Our proof of the lemma is based on an idea of Sikorav \cite[Section 8.4]{Si}.

Denote $ \ell = \floor{\frac{m}{2}} $ and $ \ell' =  \floor{\frac{m-1}{2}} $. By slightly decreasing the disc $ \cD_0 $ and re-defining $ \cD_j := \phi_j(\cD_0) $ for all $ j $, we may assume that they all are topological discs with smooth boundaries and with mutually disjoint closures. We claim that there exists a Hamiltonian diffeomorphism $ \Psi \in \Ham(M,\omega) $ such that:
\begin{itemize}
 \item $ \Psi \circ \phi_{2i} = \phi_{2i+1} $ for all $ 0 \leqslant i \leqslant \ell' $ (here $ \phi_0 : \cD_0 \rightarrow \cD_0 $ stands for the identity map).
 \item $ \| \Psi \|_\H < \eps/2 $.
\end{itemize}
To see this, for each $ 0 \leqslant i \leqslant \ell' $ find a topological open disc $ \widehat \cD_i $ that compactly contains both $ \cD_{2i} $ and $ \cD_{2i+1} $, such that all the discs $ \widehat \cD_i $ are pairwise disjoint. It is well known that for each $ \widehat \cD_i $ we can find a Hamiltonian diffeomorphism $ \Psi_i \in \Ham(M,\omega) $ satisfying $ \Psi_i \circ \phi_{2i} = \phi_{2i+1} $, whose generating Hamiltonian is normalized, compactly supported in $ \widehat \cD_i $, and has $ L^{(1,\infty)} $ norm less than $ \eps/2 $. Now denote $ \Psi = \Psi_0 \cdots \Psi_\ell $. Similarly, we can find a Hamiltonian diffeomorphism $ \Psi' \in \Ham(M,\omega) $ such that:
\begin{itemize}
 \item $ \Psi' \circ \phi_{2i-1} = \phi_{2i} $ for all $ 1 \leqslant i \leqslant \ell $.
 \item $ \| \Psi' \|_\H < \eps/2 $.
\end{itemize}

For convenience we now introduce some $ g_0, g_1, \ldots, g_\ell \in \Ham(M,\omega) $ as follows. If $ m = 2 \ell + 1 $ is odd, then define $ g_i = f_{2i} \Psi^{-1} f_{2i+1} \Psi =  f_{2i} (\phi_{2i+1} \phi_{2i}^{-1})^* f_{2i+1} $ for $ 0 \leqslant i \leqslant \ell $. If $ m = 2 \ell $ is even, then we put $ g_i =  f_{2i} \Psi^{-1} f_{2i+1} \Psi =  f_{2i} (\phi_{2i+1} \phi_{2i}^{-1})^* f_{2i+1} $ for $ 0 \leqslant i \leqslant \ell-1 $, and moreover $ g_\ell = f_{2\ell} = f_m $. Then $ \supp(g_i) \subset \cD_{2i} $ for $ 0 \leqslant i \leqslant \ell $. Denote $ \widetilde \Phi := g_0 g_1 \ldots g_\ell $. Then we have 
$$ \Phi^{-1} \widetilde \Phi = \prod_{i=0}^{\ell'} (\Psi^{-1} f_{2i+1} \Psi  f_{2i+1}^{-1}) = \left (\prod_{i=0}^{\ell'} f_{2i+1} \right)^{-1} \Psi^{-1} \left( \prod_{i=0}^{\ell'} f_{2i+1} \right) \Psi ,$$
and hence 
\begin{equation} \label{eq:Hofer-est-1}
d_\H(\Phi,\widetilde \Phi) = \| \Phi^{-1} \widetilde \Phi \|_\H \leqslant 2 \| \Psi \|_\H < \eps . 
\end{equation}
%Similarly, for an even $ m = 2 \ell $ we have 
%$$ \Phi_1 \Phi^{-1} = \prod_{i=0}^{\ell-1} (\Psi^{-1} f_{2i+1} \Psi  f_{2i+1}^{-1}) = \left (\prod_{i=1}^{\ell-1} f_{2i+1} \right)^{-1} \Psi^{-1} \left( \prod_{i=1}^{\ell-1} f_{2i+1} \right) \Psi ,$$
%and we get $(\ref{eq:Hofer-est-1})$ as well.

Define $ \hat g_i := \phi_{2i}^*g_i $ for $ 0 \leqslant i \leqslant \ell $. Then put $ \hat h_i := \prod_{j=i}^\ell \hat g_j $ for $ 0 \leqslant i \leqslant \ell $, and then define $ h_{2i} = (\phi_{2i})_* \hat h_i $ for $ 0 \leqslant i \leqslant \ell $ and
$ h_{2i-1} = (\phi_{2i-1})_* \hat h_i^{-1} $ for $ 1 \leqslant i \leqslant \ell $. Put $ \widehat \Phi := h_0 h_1 \cdots h_{2\ell} $. Then 
$$ \widetilde \Phi^{-1} \widehat \Phi = \left( \prod_{i=1}^\ell h_{2i-1} \right) \Psi^{-1} \left( \prod_{i=1}^\ell h_{2i-1} \right)^{-1} \Psi $$
and hence 
\begin{equation} \label{eq:Hofer-est-2}
d_\H(\widetilde \Phi,\widehat \Phi) = \| \widetilde \Phi^{-1} \widehat \Phi \|_\H \leqslant 2 \| \Psi \|_\H < \eps . 
\end{equation}

Finally, note that $$ h_0 = \hat h_0 = \prod_{j=0}^\ell \hat g_j = \prod_{i=0}^m \phi_i^* f_i = \Phi' ,$$
therefore 
$$ (\Phi')^{-1} \widehat \Phi = h_0^{-1} \widehat \Phi = \prod_{i=1}^{2\ell} h_i = \left( \prod_{i=1}^\ell h_{2i-1} \right) \Psi' \left( \prod_{i=1}^\ell h_{2i-1} \right)^{-1} (\Psi')^{-1} $$
and so we get 
\begin{equation} \label{eq:Hofer-est-3}
d_\H(\widehat \Phi,\Phi') = \| (\Phi')^{-1} \widehat \Phi \|_\H \leqslant 2 \| \Psi' \|_\H < \eps . 
\end{equation}

Now, the inequalities $(\ref{eq:Hofer-est-1})$, $(\ref{eq:Hofer-est-2})$ and $(\ref{eq:Hofer-est-3})$ imply $(\ref{eq:statement-Sikorav-style})$, and this concludes the proof.

\subsection{Additional remarks}

\subsubsection{An adaptation of Lemma \ref{lemma:Sikorav-style} to $ \widetilde \Ham(M,\omega) $} \label{subsubsection:Sikorav-for-universal-cover}
Lemma \ref{lemma:Sikorav-style} holds also for the universal cover $ \widetilde \Ham(M,\omega) $. More precisely, let $ (M,\omega) $, $ \eps > 0 $, discs $ \cD_0, \ldots, \cD_m \subset M $ and symplectic diffeomorphisms $ \phi_j : \cD_0 \rightarrow \cD_j $ be as in the statement of Lemma \ref{lemma:Sikorav-style}. Furthermore, let $ (f_j^t)_{t \in [0,1]} $ be Hamiltonian flows on $ M $, ($ 0 \leqslant j \leqslant m) $, such that the flow $ (f_j^t) $ is 
compactly supported in $ \cD_j $ for every $ j $. Define the Hamiltonian flows $ (\Phi_t) $, $ (\Phi_t') $ by 
$$ \Phi_t = f_0^t f_1^t \cdots f_m^t, $$ $$ \Phi_t' = f_0^t \Pi_{j=1}^m \phi_j^* f_j^t  .$$ Then for the representatives $ [(\Phi_t)],[(\Phi_t')] \in \widetilde \Ham(M,\omega) $ of the flows $ (\Phi_t), (\Phi_t') $ we have
\begin{equation} \label{eq:statement-Sikorav-style-ucover}
d_\H([(\Phi_t)],[(\Phi_t')]) < 3 \eps .
\end{equation}

Our proof of Lemma \ref{lemma:Sikorav-style} transfers almost verbatim to that case. As before, we pick a Hamiltonian diffeomorphisms $ \Psi $ and $ \Psi' $ with the same properties, and define 
flows $ (g_i^t) $ in an analogous way. Namely, if $ m = 2 \ell + 1 $ is odd then $ g_i^t = f_{2i}^t \Psi^{-1} f_{2i+1}^t \Psi $ for $ 0 \leqslant i \leqslant \ell $, and if $ m = 2 \ell $ is even then 
$ g_i^t =  f_{2i}^t \Psi^{-1} f_{2i+1}^t \Psi $ for $ 0 \leqslant i \leqslant \ell-1 $ and moreover $ g_\ell^t = f_{2\ell}^t = f_m^t $. Then we define the flow $ \widetilde \Phi_t := g_0^t g_1^t \cdots g_\ell^t $, and we get  
$$ (\Phi_t)^{-1} \widetilde \Phi_t = \left (\prod_{i=0}^{\ell'} f_{2i+1}^t \right)^{-1} \Psi^{-1} \left( \prod_{i=0}^{\ell'} f_{2i+1}^t \right) \Psi .$$
At this point we wish to conclude that we have 
\begin{equation}  \label{eq:Hofer-est-fcov}
d_\H([(\Phi_t)],[(\widetilde \Phi_t)]) = \| [((\Phi_t)^{-1} \widetilde \Phi_t)] \|_\H < \eps . 
\end{equation}
Let us explain why we are able to do that. Denote
$$ \Theta_t = \prod_{i=0}^{\ell'} f_{2i+1}^t $$
and choose a Hamiltonian flow $ (\Psi_t) $ of Hofer length less than $ \eps/2 $, such that $ \Psi_1 = \Psi $. Then the Hofer length of the Hamiltonian flow $ (\Theta_1^{-1} \Psi_t^{-1} \Theta_1 \Psi_t)_{t \in [0,1]} $ is less than $ \eps $. Hence in order to show the inequality 
$$ d_\H([(\Phi_t)],[(\widetilde \Phi_t)]) = \| [((\Phi_t)^{-1} \widetilde \Phi_t)] \|_\H = \| [(\Theta_t^{-1} \Psi_1^{-1} \Theta_t \Psi_1)] \|_H < \eps $$ it would be enough to verify that $$ [(\Theta_t^{-1} \Psi_1^{-1} \Theta_t \Psi_1)] = [(\Theta_1^{-1} \Psi_t^{-1} \Theta_1 \Psi_t)] \in \widetilde \Ham(M,\omega) .$$
But that is a general fact which holds for universal covers of Lie groups, and it follows from looking at the homotopy $ F_{s,t} := \Theta_t^{-1} \Psi_s^{-1} \Theta_t \Psi_s $, and from observing that
$ F_{s,0} = F_{0,t} = \Id $. 

We have shown the inequality $(\ref{eq:Hofer-est-fcov})$. Then, in a completely analogous way, we define the flows $ (h_i^t) $ and then define the flow $ (\widehat \Phi_t) $, showing that 
$ d_\H([(\widetilde \Phi_t)],([(\widehat \Phi_t)])  < \eps $ and $ d_\H([(\widehat \Phi_t)],([(\Phi_t')])  < \eps $, and this implies $(\ref{eq:statement-Sikorav-style-ucover})$, finishing the proof.

\subsubsection{Some properties of $ l^\infty / c_0 $} \label{subsubsection:Some-properties}

Let us explain how one can isometrically embed the normed group $ (l^\infty ,  \| \cdot \|_\infty) $ into the normed group $ (l^\infty / c_0 ,  \| \cdot \|_\infty) $, and moreover isometrically embed 
$ (C(X), \| \cdot \|_\infty) $ into $ (l^\infty ,  \| \cdot \|_\infty) $ (for a separable topological space $ X $). Decompose the set of indices $ k \in \mathbb{N} $ into a countable union of countable sets: $ \mathbb{N} = I_1 \cup I_2 \cup \cdots $. Then define the map $ \alpha: l^\infty \rightarrow l^\infty $ by $ \alpha((s_k)) = (s_i') $ where $ s_i' = s_k $ when $ i \in I_k $. The composition of $ \alpha $ with the natural homomorphism $ l^\infty \rightarrow l^\infty / c_0 $ is an isometric embedding of $ l^\infty $ into $ l^\infty / c_0 $. Now, if $ X $ is a separable topological space and $ (x_i)_{i \in \mathbb{N}} $ is a dense sequence in $ X $, then the map $ C(X) \rightarrow l^\infty $ which sends $ f \in C(X) $ to the sequence $ (f(x_i))_{i\in\mathbb{N}} $, is an isometric embedding. 

Now we explain why there is no isometric group embedding of $ (l^\infty / c_0 ,  \| \cdot \|_\infty) $ into $ (l^\infty ,  \| \cdot \|_\infty) $. Assume on the contrary that such an embedding $ \iota : l^\infty / c_0 \rightarrow l^\infty $ exists. Consider the set $ J $ consisting of all infinite sequences $ \alpha = (\alpha_0,\alpha_1, \ldots) $ where $ \alpha_i \in \{1,2 \} $ for each $ i $. For every $ \alpha \in J $ define a subset $ I_\alpha $ consisting of all integers of the form $ \sum_{i=0}^k \alpha_i 3^i $ for all $ k \geqslant 0 $, and then define $ s^\alpha = [(s^\alpha_j)] \in l^\infty / c_0 $ where $ s^\alpha_j = 1 $ for $ j \in I_\alpha $ and $ s^\alpha_j = 0 $ otherwise. For every distinct $ \alpha, \beta \in J $, the intersection $ I_\alpha \cap I_\beta $ is finite. Therefore for every $ \alpha_1, \ldots, \alpha_m \in J $ and $ t_1,\ldots,t_m \in \R $ we have 
$$ \| t_1 s^{\alpha_1} + \ldots + t_m s^{\alpha_m} \|_\infty = \max_\ell | t_\ell | .$$
Now for each pair of integers $ i, k \geqslant 1 $ consider the subset $ J_{i,k} \subset J $ that consists of all $ \alpha \in J $ for which the absolute value of the $ i $-th coordinate of $ \iota(s^\alpha) \in l^\infty $ is greater than $ 1/k $. Then $ J_{i,k} $ is a finite set of cardinality $ | J_{i,k} | < k $. Indeed, otherwise taking some distinct $ \alpha_1, \ldots, \alpha_k \in J_{i,k} $, for a suitable choice of $ \epsilon_i \in \{ -1,1 \} $, the $ i $-th coordinate of $ \sum_{j=1}^k \epsilon_j \iota(s^{\alpha_j}) $ is greater than $ 1 $ and consequently 
$$ \| \iota( \sum_{j=1}^k \epsilon_j s^{\alpha_j} ) \|_\infty = \| \sum_{j=1}^k \epsilon_j \iota(s^{\alpha_j}) \|_\infty > 1 ,$$
which contradicts $ \| \sum_{j=1}^k \epsilon_j s^{\alpha_j}  \|_\infty = 1 $. This shows that $ J_{i,k} $ is a finite set and hence the union $ \cup_{i,k} J_{i,k} $ is countable. Since $ J $ is uncountable we therefore can find some $ \alpha \in J $ that does not belong to any of $ J_{i,k} $, which means that $ s^\alpha $ belongs to the kernel of $ \iota $, and so the kernel is non-trivial as stated. A similar argument shows that every continuous group homomorphism between $ (l^\infty / c_0 ,  \| \cdot \|_\infty) $ and $ (l^\infty ,  \| \cdot \|_\infty) $ has a non-trivial kernel (in fact, such a homomorphism is always linear over $ \R $ and the kernel must be an infinite dimensional linear subspace of $ l^\infty / c_0 $).

\bigskip
\noindent Lev Buhovski\\
School of Mathematical Sciences, Tel Aviv University \\
{\it e-mail}: levbuh@tauex.tau.ac.il
\bigskip

\end{document}